\documentclass{article}
\usepackage[a4paper, margin=1.5in]{geometry}  
\usepackage[T1]{fontenc}      
\usepackage[utf8]{inputenc}   
\DeclareUnicodeCharacter{00A0}{ } 
\usepackage{microtype}        
\usepackage{bbm}              
\usepackage{amsmath}
\usepackage{amsthm}
\usepackage{amssymb}
\usepackage{commath}          
\usepackage{mathtools}
\usepackage{thmtools}
\usepackage[authoryear,round]{natbib}
\usepackage[capitalise]{cleveref}
\usepackage{verbatim}               
\usepackage{graphicx}
\usepackage[dvipsnames]{xcolor}     

\usepackage{tikz}
\usepackage{pgfplots}

\usepackage{subcaption}
\usepackage[font=footnotesize,labelfont=bf]{caption}
\pgfplotsset{compat=newest}
\pgfplotsset{plot coordinates/math parser=false}

\usetikzlibrary{decorations.pathreplacing}
\usetikzlibrary{patterns}
\usepgfplotslibrary{groupplots}
\usetikzlibrary{plotmarks}
\usetikzlibrary{calc}

\usetikzlibrary{external}
\newlength{\figwidth}
\newlength{\figheight}
\definecolor{gray1}{gray}{0.0}
\definecolor{gray2}{gray}{0.25}
\definecolor{gray3}{gray}{0.5}
\definecolor{gray4}{gray}{0.7}
\definecolor{gray5}{gray}{0.9}

\newlength{\prel}

\pgfplotsset{
  title style = {font=\small},
}

\numberwithin{equation}{section}
\declaretheorem[Refname={Theorem,Theorems}]{theorem}
\numberwithin{theorem}{section} 

\declaretheorem[style=definition,numberlike=theorem,Refname={Remark,Remarks}]{remark}
\declaretheorem[numberlike=theorem,Refname={Lemma,Lemmas}]{lemma}
\declaretheorem[name=Corollary,numberlike=theorem,Refname={Corollary,Corollaries}]{corollary}
\declaretheorem[name=Proposition,numberlike=theorem,Refname={Proposition,Propositions}]{proposition}

\newcommand{\rkhs}{\mathcal{H}} 
\DeclareMathOperator{\e}{e} 
\renewcommand{\b}[1]{\pmb{#1}} 
\newcommand{\T}{\mathsf{T}} 

\DeclarePairedDelimiterX\Set[2]{\lbrace}{\rbrace}%
{ #1 \,:\, #2 }                                         
\DeclarePairedDelimiterX\inprod[2]{\langle}{\rangle}%
{ #1 , #2 }                                             
\DeclarePairedDelimiter\floor{\lfloor}{\rfloor}         

\DeclareMathOperator*{\argmin}{arg\,min}  

\newcommand{\R}{\mathbb{R}} 
\newcommand{\N}{\mathbb{N}} 

\newcommand{\textsm}[1]{\text{\tiny{\textup{#1}}}}

\title{\textbf{Integration in reproducing kernel Hilbert spaces of Gaussian kernels}}
\author{
  Toni Karvonen\textsuperscript{1}, Chris J.\ Oates\textsuperscript{2,1} and Mark Girolami\textsuperscript{3,1}
  \vspace{0.4cm} \\ \emph{\textsuperscript{1}The Alan Turing Institute, United Kingdom}
  \vspace{0.2cm} \\ \emph{\textsuperscript{2}School of Mathematics, Statistics \& Physics} \\ \emph{Newcastle University, United Kingdom}
  \vspace{0.2cm} \\ \emph{\textsuperscript{3}Department of Engineering} \\ \emph{University of Cambridge, United Kingdom}
}

\begin{document}

\maketitle

\begin{abstract}
  \noindent
  The Gaussian kernel plays a central role in machine learning, uncertainty quantification and scattered data approximation, but has received relatively little attention from a numerical analysis standpoint.
  The basic problem of finding an algorithm for efficient numerical integration of functions reproduced by Gaussian kernels has not been fully solved.
  In this article we construct two classes of algorithms that use $N$ evaluations to integrate $d$-variate functions reproduced by Gaussian kernels and prove the exponential or super-algebraic decay of their worst-case errors. In contrast to earlier work, no constraints are placed on the length-scale parameter of the Gaussian kernel.
  The first class of algorithms is obtained via an appropriate scaling of the classical Gauss--Hermite rules.
  For these algorithms we derive lower and upper bounds on the worst-case error of the forms $\exp(-c_1 N^{1/d}) N^{1/(4d)}$ and $\exp(-c_2 N^{1/d}) N^{-1/(4d)}$, respectively, for positive constants $c_1 > c_2$.
  The second class of algorithms we construct is more flexible and uses worst-case optimal weights for points that may be taken as a nested sequence.
  For these algorithms we derive upper bounds of the form $\exp(-c_3 N^{1/(2d)})$ for a positive constant~$c_3$.
\end{abstract}

\section{Introduction} \label{sec:introduction}

This article considers numerical approximation of a $d$-dimensional Gaussian integral
\begin{equation} \label{eq:integral}
  I_{\b{\alpha}}(f) \coloneqq \int_{\R^d} f(\b{x}) \Bigg[ \prod_{i=1}^d \frac{1}{\sqrt{2\pi} \alpha_i } \exp\bigg(\! -\frac{x_i^2}{2\alpha_i^2} \bigg) \Bigg] \dif \b{x},
\end{equation}
where the integrand $f \colon \R^d \to \R$ is assumed to belong to $\mathcal{H}(K_{\b{\ell}})$, the reproducing kernel Hilbert space (RKHS) of the symmetric positive-definite Gaussian kernel
\begin{equation} \label{eq:gauss-kernel}
  K_{\b{\ell}}(\b{x}, \b{y}) \coloneqq \prod_{i=1}^d K_{\ell_i}(x_i, y_i), \quad\quad K_\ell(x,y) \coloneqq \exp\bigg(\! -\frac{(x-y)^2}{2\ell^2}\bigg),
\end{equation}
where elements of both the variance parameter $\b{\alpha} = (\alpha_1,\dots,\alpha_d)$ and the length-scale parameter $\b{\ell} = (\ell_1, \ldots, \ell_d)$ are positive.
The inner product and norm of $\mathcal{H}(K_{\b{\ell}})$ are denoted $\inprod{\cdot}{\cdot}_{\b{\ell}}$ and $\norm[0]{\cdot}_{\b{\ell}}$.
The Gaussian kernel and its RKHS are commonly used in machine learning~\citep{RasmussenWilliams2006,Steinwart2008}, uncertainty quantification~\citep{Sullivan2015} and scattered data approximation~\citep{Wendland2005,FasshauerMcCourt2015}.
The quality of an integration rule $Q_n(f) \coloneqq \sum_{i=1}^n w_{i} f(\b{x}_{i})$, having points $\b{x}_i \in \R^d$ and weights $w_i \in \R$, for integration of functions in the RKHS can be measured in terms of its worst-case error
\begin{equation} \label{eq:wce-def}
  e_{\b{\alpha}, \b{\ell}}( Q_n ) \coloneqq \sup_{ \norm[0]{f}_{ \b{\ell} } \leq 1} \abs[0]{ I_{\b{\alpha}}(f) - Q_n(f) } = \norm[0]{ \mathcal{I}_{\b{\alpha},\b{\ell}} - \mathcal{Q}_{\b{\ell},n} }_{\b{\ell}},
\end{equation}
where the functions $\mathcal{I}_{\b{\alpha},\b{\ell}}$ and $\mathcal{Q}_{\b{\ell},n}$ are the Riesz representers of the linear functionals $I_{\b{\alpha}}$ and $Q_n$, meaning that $I_{\b{\alpha}}(f) = \inprod{f}{\mathcal{I}_{\b{\alpha},\b{\ell}}}_{\b{\ell}}$ and $Q_n(f) = \inprod{f}{\mathcal{Q}_{\b{\ell},n}}_{\b{\ell}}$ for any $f \in \mathcal{H}(K_{\b{\ell}})$.
By the reproducing property of the kernel $K_{\b{\ell}}$ the representers can be computed pointwise as~\citep[e.g.,][Proposition~3.5]{Oettershagen2017}
\begin{equation} \label{eq:representers-explicit}
  \mathcal{I}_{\b{\alpha},\b{\ell}}(\b{x}) = I_{\b{\alpha}} ( K_{\b{\ell}}(\cdot, \b{x}) ) \quad \text{ and } \quad \mathcal{Q}_{\b{\ell},n}(\b{x}) = \sum_{i=1}^n w_i K_{\b{\ell}}(\b{x}_i, \b{x}).
\end{equation}
  Because $\norm[0]{\mathcal{I}_{\b{\alpha},\b{\ell}}}_{\b{\ell}}^2 = \inprod{\mathcal{I}_{\b{\alpha},\b{\ell}}}{\mathcal{I}_{\b{\alpha},\b{\ell}}}_{\b{\ell}} = I_{\b{\alpha}}(\mathcal{I}_{\b{\alpha},\b{\ell}})$ and $\norm[0]{\mathcal{Q}_{\b{\ell},n}}_{\b{\ell}}^2 = \inprod{\mathcal{Q}_{\b{\ell},n}}{\mathcal{Q}_{\b{\ell},n}}_{\b{\ell}} = Q_{n}(\mathcal{Q}_{\b{\ell},n})$ by definitions of the representers, expansion of the RKHS norm in~\eqref{eq:wce-def} and~\eqref{eq:representers-explicit} yields
\begin{equation} \label{eq:wce-explicit}
  \begin{split}
    e_{\b{\alpha}, \b{\ell}}(Q_n) ={}& \sqrt{ \norm[0]{\mathcal{I}_{\b{\alpha},\b{\ell}}}_{\b{\ell}}^2- 2 \inprod{\mathcal{I}_{\b{\alpha}}}{ \mathcal{Q}_{\b{\ell},n} }_{\b{\ell}} + \norm[0]{\mathcal{Q}_{\b{\ell},n}}_{\b{\ell}}^2 } \\
    ={}& \sqrt{ I_{\b{\alpha}}^{\b{x}} I_{\b{\alpha}}^{\b{y}}( K_{\b{\ell}}(\b{x}, \b{y})) - 2 \sum_{i=1}^n w_i I_{\b{\alpha}}( K_{\b{\ell}}(\cdot, \b{x}_i)) + \sum_{i=1}^n \sum_{j=1}^n w_i w_j K_{\b{\ell}}(\b{x}_i, \b{x}_j) },
  \end{split}
\end{equation}
where the superscript in $I_{\b{\alpha}}^{\b{x}}$ indicates that integration is to be performed with respect to the dummy variable $\b{x}$ in \eqref{eq:wce-explicit}.
By means of the worst-case error, the integration error for any $f \in \mathcal{H}(K_{\b{\ell}})$ can be decomposed as follows:
\begin{equation} \label{eq:error-decoupling}
  \abs[0]{ I_{\b{\alpha}}(f) - Q_n(f) } \leq \norm[0]{f}_{ \b{\ell} } e_{\b{\alpha}, \b{\ell}} ( Q_n ).
\end{equation}

The only prior work containing bounds on the worst-case errors in this setting appears to be due to \citet{KuoWozniakowski2012} and \citet{KuoSloanWozniakowski2017}. When $d=1$ they consider the $n$-point Gauss--Hermite rule $Q_{\alpha,n}^\textsm{GH}$, which satisfies $Q_{\alpha,n}^\textsm{GH}(f) = I_\alpha(f)$ whenever~$f$ is a polynomial of degree at most $2n-1$, and prove that
\begin{equation*}
  e_{\alpha, \ell} ( Q_{\alpha,n}^\textsm{GH} ) \leq b_n \bigg( \frac{ \alpha^2 }{ \ell^2 } \bigg)^{n},
\end{equation*}
with $(b_n)_{n=1}^\infty$ a decreasing sequence converging to $2^{-1/4}$.
That is, the Gauss--Hermite rule converges with an exponential rate if $\ell > \alpha$.\footnote{Note that the matching lower bound claimed in \citet{KuoWozniakowski2012} is erroneous as pointed out by \citet[p.\@~830]{KuoSloanWozniakowski2017}.} Their potential non-convergence when~$\ell \leq \alpha$ is perhaps not surprising because these rules are not adapted to the RKHS, and in particular to the length-scale parameter.
Tensor product extensions for the multivariate case are also available, with similar constraints on $\alpha_i$ and $\ell_i$.
\citet{KarvonenSarkka2019} propose using certain scaled versions of Gauss--Hermite rules but are unable to prove the convergence of their rules, their error estimates being dependent on the sum of absolute values of the weights.
Approximation, measured in the $L^2$-norm corresponding to~\eqref{eq:integral}, is analysed in worst-case setting in \citet{FasshauerHickernell2012} and \citet{SloanWozniakowski2018} and in average-case setting in \citet{FasshauerHickernell2010} and \citet{ChenWang2019}.
Techniques similar to those used here have been used by \citet{Irrgeher2015,Irrgeher2016} and \citet{Dick2018} to analyse integration and approximation of functions in Hermite spaces whose reproducing kernels admit expansions in terms of Hermite polynomials.

\subsection{Contributions}

We develop kernel-dependent integration rules that are \emph{parameter-universal} in that they provably converge for all values of the variance and scale parameters $\b{\alpha}$ and $\b{\ell}$:
\begin{itemize}
\item In Section~\ref{sec:scaled-GH} we consider appropriately scaled versions of Gauss--Hermite rules and their tensor products. In the univariate case these rules, denoted $Q_{\alpha,\ell,n}^\textsm{GH}$, satisfy
  \begin{equation*}
    C_{1} \bigg( \frac{\alpha^2}{ 2(\alpha^2 + \ell^2)} \bigg)^n n^{1/4} \leq e_{\alpha,\ell}( Q_{\alpha, \ell, n}^\textsm{GH} ) \leq C_{2} \bigg( \frac{\alpha^2}{ \alpha^2 + \ell^2} \bigg)^n n^{-1/4}
  \end{equation*}
  for any $n \geq 1$ and certain positive constants $C_1$ and $C_2$, which shows that the rules enjoy exponential convergence for any values of $\alpha$ and $\ell$.
  See Theorem~\ref{thm:gh-1d} for details when $d=1$ and Theorem~\ref{thm:gh-tensor} and Corollary~\ref{cor:gh-tensor-isotropic} for the multivariate case.
  The rules are related to those developed by \citet{KarvonenSarkka2019}, but much simpler and more amenable to error analysis.
\item In Section~\ref{sec:kernel-quadrature} we consider potentially nested integration rules with worst-case optimal weights~\citep{Larkin1970,Oettershagen2017}, often known as kernel quadrature rules or, in statistical literature, Bayesian quadrature rules~\citep{Briol2019}.
  For a point set $X \subset \R$ such a rule is denoted $Q_{\alpha,\ell,X}^\textsm{opt}$.
  After decomposing the unbounded integration domain into bounded sub-domains and a ``tail domain'' we apply results from scattered data approximation literature~\citep{RiegerZwicknagl2010} to each of the bounded sub-domains and thereafter sum the individual errors.
  For a specific sequence $(X_k)_{k=1}^\infty$ of point sets, each containing $n = k(k+1)$ points, we compute
  \begin{equation*}
    e_{\alpha,\ell}( Q_{\alpha,\ell,X_k}^\textsm{opt} ) \leq C  \exp\bigg(\!-\frac{\sqrt{n}}{2\sqrt{2} \alpha^2} \bigg) 
  \end{equation*}
  for a positive constant $C$ and for any sufficiently large $k \geq 1$.
  The main results for $d=1$ are Proposition~\ref{prop:kernel-univariate-generic} and Theorem~\ref{thm:kq-univariate} while the multivariate case is contained in Theorem~\ref{thm:kq-tensor} and Corollary~\ref{cor:kq-tensor}.
The domain decomposition technique we use has been inspired by the method in \citet{Dick2018} and \citet[Chapter~6]{Suzuki2020}.
\end{itemize}

\subsection{Hilbert space of the Gaussian kernel} \label{sec:RKHS}

Before proceeding we review some facts about the Hilbert space $\mathcal{H}(K_{\b{\ell}})$.
Recall that any symmetric positive-definite kernel $K \colon \Omega \times \Omega \to \R$ on a set $\Omega$ induces a unique RKHS $\mathcal{H}(K)$ consisting of real-valued functions defined on $\Omega$ and equipped with an inner product $\inprod{\cdot}{\cdot}$ and the associated norm $\norm[0]{\cdot}$.
For any $x \in \Omega$, the function $K(\cdot,x)$ is in $\mathcal{H}(K)$ and the kernel has the reproducing property: $\inprod{f}{K(\cdot,x)} = f(x)$ for any $f \in \mathcal{H}(K)$ and $x \in \Omega$.

To describe the structure of the RKHS $\mathcal{H}(K_{\b{\ell}})$ of the Gaussian kernel~\eqref{eq:gauss-kernel} we make use of a simple orthonormal basis from \citet{Steinwart2006}, \citet{DeMarchiSchaback2009} and \citet{Minh2010}.
Consider first the case $d=1$ and define
\begin{equation} \label{eq:basis}
  \phi_{\ell,m}(x) \coloneqq \frac{1}{\ell^m \sqrt{m!}} \, x^m \exp\bigg(\! -\frac{x^2}{2\ell^2} \bigg).
\end{equation}
Then it follows~\citep{Minh2010} from the expansion $K_\ell(x,y) = \sum_{m=0}^\infty \phi_{\ell,m}(x) \phi_{\ell,m}(y)$ that $\{\phi_{\ell,m}\}_{m=0}^\infty$ is an orthonormal basis of $\mathcal{H}(K_\ell)$ and consequently
\begin{equation} \label{eq:gauss-RKHS}
  \mathcal{H}(K_\ell) = \Set[\bigg]{ f = \sum_{m=0}^\infty f_m \phi_{\ell,m} }{ \norm[0]{f}_\ell^2 \coloneqq \sum_{m=0}^\infty f_m^2 < \infty}.
\end{equation}
Because the multivariate Gaussian kernels~\eqref{eq:gauss-kernel} we consider are products of univariate kernels, the RKHS $\mathcal{H}(K_{\b{\ell}})$ of $d$-variate functions is the tensor product of the univariate spaces $\mathcal{H}(K_{\ell_i})$ for $i=1, \ldots, d$~\citep[p.\@~31]{BerlinetThomasAgnan2004}.
Moreover, the functions
\begin{equation} \label{eq:phi-d}
  \phi_{\b{\ell}, \b{m} }(\b{x}) \coloneqq \frac{1}{\b{\ell}^{\b{m}} \sqrt{\b{m}!}} \b{x}^{\b{m}} \prod_{i=1}^d \exp\bigg(\!- \frac{x_i^2}{2\ell_i^2} \bigg)
\end{equation}
for $\b{m} \in \N_0^d$ form an orthonormal basis of $\mathcal{H}(K_{\b{\ell}})$. Here $\N_0^d$ is the collection of $d$-dimensional non-negative multi-indices (and later $\N^d$ will be that of positive multi-indices), $\b{m}! \coloneqq m_1 ! \times \cdots \times m_d!$ and $\b{x}^{\b{m}} \coloneqq x_1^{m_1} \times \cdots \times x_d^{m_d}$ for $\b{m} \in \N_0^d$ and $\b{x} \in \R^d$.

The prior work~\citep{KuoWozniakowski2012,KuoSloanWozniakowski2017,KarvonenSarkka2019} on integration in the Gaussian RKHS is based on the Mercer basis functions
\begin{equation} \label{eq:mercer-basis}
    \varphi_{\alpha,\ell,m}(x) \coloneqq \sqrt{\frac{b_{\alpha,\ell}}{m!}} \exp(-c_{\alpha,\ell}^2 x^2) \mathrm{H}_m\bigg( \frac{b_{\alpha,\ell}x}{\alpha}\bigg),
\end{equation}
where $b_{\alpha,\ell}$ and $c_{\alpha,\ell}$ are certain constants
and $\mathrm{H}_m$ are the probabilists' Hermite polynomials, to be defined in~\eqref{eq:hermite}.
The functions $\varphi_{\alpha,\ell,m}$ have the $L^2$-orthonormality property
\begin{equation} \label{eq:mercer-L2}
    \frac{1}{\sqrt{2\pi} \alpha} \int_\R \varphi_{\alpha,\ell,p}(x) \varphi_{\alpha,\ell,q}(x) \exp\bigg(\!-\frac{x^2}{2\alpha^2} \bigg) \dif x = \delta_{pq}
\end{equation}
while $\{ \lambda_{\alpha,\ell,m}^{1/2} \varphi_{\alpha,\ell,m} \}_{m=0}^\infty$, for an exponentially decaying positive sequence $(\lambda_{\alpha,\ell,m})_{m=0}^\infty$, is an orthonormal basis of $\mathcal{H}(K_\ell)$ and $K_\ell(x,y) = \sum_{m=0}^\infty \lambda_{\alpha,\ell,m} \varphi_{\alpha,\ell,m}(x) \varphi_{\alpha,\ell,m}(y)$.
It seems to us that in many situations the simpler basis~\eqref{eq:basis} ought to be preferred over the Mercer basis~\eqref{eq:mercer-basis}.
As will become evident in Section~\ref{sec:scaled-GH}, the $L^2$-orthonormality~\eqref{eq:mercer-L2} of the Mercer basis is not necessary for analysing integration error.

\section{Scaled Gauss--Hermite rules} \label{sec:scaled-GH}

In this section we introduce an appropriate RKHS-dependent scaling for Gauss--Hermite rules and their tensor product extensions. The use of this scaling guarantees exponential convergence for all values of the variance and length-scale parameters $\b{\alpha}$ and $\b{\ell}$.

\subsection{Gauss--Hermite quadrature}

In one dimension, the $n$-point (generalised) Gauss--Hermite rule 
\begin{equation*}
  Q_{\alpha,n}^\textsm{GH}(f) \coloneqq \sum_{i=1}^n w_{n,i}^\textsm{GH} f( \alpha x_{n,i}^\textsm{GH})
\end{equation*}
approximates the integral $I_{\alpha}(f)$ and is uniquely characterised by the property that
\begin{equation*}
  Q_{\alpha,n}^\textsm{GH}(p) = I_\alpha(p) \quad \text{ for every polynomial $p$ of degree at most $2n-1$}.
\end{equation*}
Its points are obtained by scaling $x_{n,i}^\textsm{GH}$, the roots of the $n$th Hermite polynomial
\begin{equation} \label{eq:hermite}
  \mathrm{H}_n(x) \coloneqq (-1)^n \e^{x^2/2} \frac{\dif^{\,n}}{\dif x^n} \e^{-x^2/2},
\end{equation}
and the weights are
\begin{equation*}
  w_{n,i}^\textsm{GH} \coloneqq \frac{n!}{n^2 \mathrm{H}_{n-1}(x_{n,i}^\textsm{GH})}.
\end{equation*}
Note that $\sum_{i=1}^n w_{n,i}^\textsm{GH} = 1$ since the rule must be exact for constant functions. Furthermore, the point set is symmetric: for every $i \leq n$ there is $j \leq n$ such that $x_{n,i}^\textsm{GH} = -x_{n,j}^\textsm{GH}$.
In practice, the weights and points are computed with the Golub--Welsch algorithm that exploits the three-term recurrence relation of the Hermite polynomials~\citep[Section~3.1.1.1]{Gautschi2004}.
If $f$ has $2n$ continuous derivatives, then the error of the Gauss--Hermite quadrature is~\citep[Section~8.7]{Hildebrand1987}
\begin{equation} \label{eq:GH-error}
  I_{\alpha}(f) - Q_{\alpha,n}^\textsm{GH}(f) = f^{(2n)}({\xi}) \frac{\alpha^{2n} n!}{(2n)!} \quad \text{ for some } \quad \xi \in \R,
\end{equation}
which in particular implies that $Q_{\alpha,n}^\textsm{GH}$ underestimates the value of the integral if the $2n$th derivative of the integrand is everywhere positive.

\subsection{Integration of the orthonormal basis} \label{sec:scaled-GH-subsec}

Recall from Section~\ref{sec:RKHS} that
\begin{equation*}
  \phi_{\ell,m}(x)  = \frac{1}{\ell^m \sqrt{m!}} x^m \exp\bigg(\! -\frac{x^2}{2\ell^2} \bigg)
\end{equation*}
form an orthonormal basis of $\mathcal{H}(K_\ell)$. Denote
\begin{equation*}
  \psi_{\ell,m}(x) \coloneqq x^m \exp\bigg(\! -\frac{x^2}{2\ell^2} \bigg),
\end{equation*}
so that $\phi_{\ell,m} = (\ell^m \sqrt{m!})^{-1} \psi_{\ell,m}(x)$.
We now construct an $n$-point scaled Gauss--Hermite rule, $ Q_{\alpha,\ell,n}^\textsm{GH} $, such that
\begin{equation} \label{eq:SGH-conditions}
  Q_{\alpha,\ell,n}^\textsm{GH}(\phi_{\ell,m}) = I_{\alpha}(\phi_{\ell,m})
\end{equation}
for every $0 \leq m \leq 2n-1$.
Note that this is equivalent to $Q_{\alpha,\ell,n}^\textsm{GH}(\psi_{\ell,m}) = I_{\alpha}(\psi_{\ell,m})$ for every $0 \leq m \leq 2n-1$.
Let
\begin{equation*}
  \beta \coloneqq \sqrt{ \frac{\alpha^2\ell^2}{\alpha^2 + \ell^2}}.
\end{equation*}
Then each of the $2n$ exactness conditions~\eqref{eq:SGH-conditions} can be written as
\begin{equation} \label{eq:SGH-identity}
  \begin{split}
    Q_{\alpha,\ell,n}^\textsm{GH}(\psi_{\ell,m} ) = I_{\alpha}(\psi_{\ell,m}) &= \frac{1}{\sqrt{2\pi}\alpha} \int_\R x^m \exp\bigg(\! - \frac{x^2}{2\ell^2} \bigg) \exp\bigg(\! -\frac{x^2}{2\alpha^2} \bigg) \dif x \\
    &= \frac{\beta}{\alpha} \frac{1}{\sqrt{2\pi} \beta} \int_\R x^m \exp\bigg(\! - \frac{x^2}{2\beta^2} \bigg) \dif x \\
    &= \frac{\beta}{\alpha} I_{\beta}(x^m).
    \end{split}
\end{equation}
The desired quadrature rule can be thus realised as a scaled Gauss--Hermite rule for approximation of $I_{\beta}$:
\begin{equation*}
  Q_{\alpha,\ell,n}^\textsm{GH}(f) \coloneqq \frac{\beta}{\alpha} Q_{\beta,n}^\textsm{GH}(f_\textsm{exp}) \quad \text{ for } \quad f_\textsm{exp}(x) \coloneqq f(x) \exp\bigg(\frac{x^2}{2\ell^2} \bigg),
\end{equation*}
the exactness up to order $2n$ of which can be verified by observing that for $f = \psi_{\ell,m}$ we have $f_\textsm{exp}(x) = x^m$ and thus, by~\eqref{eq:SGH-identity},
\begin{equation*}
  Q_{\alpha,\ell,n}^\textsm{GH}(\psi_{\ell,m}) = \frac{\beta}{\alpha} Q_{\beta,n}^\textsm{GH}(x^m) = \frac{\beta}{\alpha} I_\beta(x^m) = I_\alpha(\psi_{\ell,m})
\end{equation*}
for every $0 \leq m \leq 2n-1$.
The scaled rule can be written as
\begin{equation*}
  Q_{\alpha,\ell,n}^\textsm{GH}(f) = \frac{\beta}{\alpha} \sum_{i=1}^n w_{n,i}^\textsm{GH} \exp\bigg( \frac{\beta^2 (x_{n,i}^\textsm{GH})^2}{2\ell^2} \bigg) f(\beta x_{n,i}^\textsm{GH}) = \sum_{i=1}^n w_{\alpha,\ell,n,i}^\textsm{GH} f(x_{\alpha,\ell,n,i}^\textsm{GH}),
\end{equation*}
the points and weights being
\begin{equation} \label{eq:scaled-GH-weights-point}
  x_{\alpha,\ell,n,i}^\textsm{GH} \coloneqq \beta x_{n,i}^\textsm{GH} \quad \text{ and } \quad w_{\alpha,\ell,n,i}^\textsm{GH} \coloneqq \frac{\beta}{\alpha} w_{n,i}^\textsm{GH} \exp\bigg( \frac{\beta^2 (x_{n,i}^\textsm{GH})^2}{2\ell^2} \bigg) > 0.
\end{equation}
This is an example of a generalised Gaussian quadrature rule, a quadrature rule that uses $n$ function evaluations to integrate exactly a collection of $2n$ functions~\citep{Barrow1978}.

\begin{remark} \label{remark:general-weights}
  Note that the above construction can be carried out for general weighted integration problems. Namely, consider the computation of $I^\nu(f) \coloneqq \int_a^b f(x) \nu(x) \dif x$ for $-\infty \leq a < b \leq \infty$ and a weight function $\nu \colon \Omega \to [0, \infty)$ that is sufficiently regular to guarantee~\citep[Section~1.1]{Gautschi2004} the existence of a Gaussian quadrature rule $Q_{n}^\nu(f) \coloneqq \sum_{i=1}^n w_{n,i}^\nu f(x_{n,i}^\nu)$ such that $Q_{n}^\nu(p) = I_\nu(p) < \infty$ for every polynomial $p$ of degree at most $2n-1$.
Define then $\bar{\nu}(x) \coloneqq \exp(-x^2/(2\ell^2)) \nu(x)$ and write
\begin{equation*}
  I^\nu(\psi_{\ell,m}) = \int_a^b x^m \exp\bigg(\!-\frac{x^2}{2\ell^2}\bigg) \nu(x) \dif x = I^{\bar{\nu}}(x^m).
\end{equation*}
It is easy to see that the quadrature rule
\begin{equation*}
  Q_{\ell,n}^\nu(f) \coloneqq \sum_{i=1}^n w_{n,i}^{\bar{\nu}} \exp\bigg(\frac{(x_{n,i}^{\bar{\nu}})^2}{2\ell^2}\bigg) f(x_{n,i}^{\bar{\nu}})
\end{equation*}
satisfies $Q_{\ell,n}^\nu(\phi_{\ell,m}) = I^\nu(\phi_{\ell,m})$ for $0 \leq m \leq 2n-1$.
The worst-case error can be bounded using the methods in Section~\ref{sec:convergence-gh}, but the bounds involve integrals of monic $\bar{\nu}$-orthogonal polynomials that appear difficult to estimate~\citep[Section~8.4]{Hildebrand1987}.
\end{remark}

\subsection{Error estimates in one dimension} \label{sec:convergence-gh}

This section establishes exponential upper and lower bounds on the worst-case error of the scaled Gauss--Hermite rule $Q_{\alpha,\ell,n}^\textsm{GH}$.
For these estimates recall Stirling's approximation
\begin{equation} \label{eq:stirling-asymptotic}
  n! \sim \sqrt{2\pi} n^{n+1/2} \e^{-n}.
\end{equation}
A version that is valid for finite $n$ is
\begin{equation} \label{eq:stirling-finite}
  \sqrt{2\pi} n^{n+1/2} \e^{-n} \leq n! \leq \e n^{n+1/2} \e^{-n}.
\end{equation}
These follow from the more precise bounds due to \citet{Robbins1955}.

\begin{lemma} \label{lemma:C} The sequence $(C_n)_{n=1}^\infty$ defined by
  \begin{equation} \label{eq:C}
    C_n \coloneqq \frac{2^n n!}{\sqrt{(2n)!}} \, n^{-1/4}
  \end{equation}
  is strictly decreasing and satisfies
\begin{equation*}
  \lim_{n \to \infty} C_{n} = \pi^{1/4} \quad \text{ and } \quad \pi^{1/4} < C_{n} \leq \frac{\e}{(2\pi)^{1/4}}.
\end{equation*}  
\end{lemma}
\begin{proof} Write
  \begin{equation*}
    \frac{C_n}{C_{n+1}} = \frac{\sqrt{(2n+1)(2n+2)}}{2(n+1)} \bigg( \frac{n+1}{n} \bigg)^{1/4} = \sqrt{\frac{2n+1}{2n+2}} \bigg( \frac{n+1}{n} \bigg)^{1/4}.
  \end{equation*}
  Because
  \begin{equation*}
    \bigg(\frac{2n+1}{2n+2}\bigg)^2 \frac{n+1}{n} = \bigg(1 - \frac{1}{2(n+1)} \bigg)^2 \frac{n+1}{n} = 1 + \frac{1}{4n(n+1)} > 1,
  \end{equation*}
  the sequence $(C_n)_{n=1}^\infty$ is strictly decreasing. Its limit is obtained from the asymptotic form~\eqref{eq:stirling-asymptotic} of Stirling's approximation and the upper bound $C_n \leq \e (2\pi)^{-1/4}$ from~\eqref{eq:stirling-finite}.
\end{proof}

\begin{lemma} \label{lemma:2n} For any $n \geq 1$ we have
  \begin{equation*}
    I_\alpha(\phi_{\ell,2n}) - Q_{\alpha,\ell,n}^\textsm{GH}(\phi_{\ell,2n}) = C_n \frac{\beta}{\alpha} \bigg( \frac{\beta^2}{2\ell^2} \bigg)^n n^{1/4},
  \end{equation*}
  where $C_n > 0$ is defined in~\eqref{eq:C}.
\end{lemma}
\begin{proof}
  Because
  \begin{equation*}
    I_{\alpha}(\psi_{\ell,m}) = \frac{\beta}{\alpha} I_{\beta}(x^m) \quad \text{ and } \quad Q_{\alpha,\ell,n}^\textsm{GH}(\psi_{\ell,m}) = \frac{\beta}{\alpha} Q_{\beta,n}^\textsm{GH}(x^m)
  \end{equation*}
  for every $m \geq 0$, we can use the Gauss--Hermite error formula~\eqref{eq:GH-error} to deduce that
  \begin{equation*}
    \begin{split}
      I_{\alpha}(\phi_{\ell,2n}) - Q_{\alpha,\ell,n}^\textsm{GH}(\phi_{\ell,2n}) = \frac{\beta}{\alpha} \frac{1}{\ell^{2n} \sqrt{(2n)!}} [ I_{\beta}(x^{2n}) - Q_{\beta,n}^\textsm{GH}(x^{2n}) ] &= \frac{\beta}{\alpha} \frac{\beta^{2n} n!}{\ell^{2n} \sqrt{(2n)!}}\\
      &= C_n \frac{\beta}{\alpha} \bigg( \frac{\beta^2}{2\ell^2} \bigg)^n n^{1/4}.
      \end{split}
  \end{equation*}
\end{proof}

\begin{lemma} \label{lemma:>2n} For any $n \geq 1$ and $q \geq 0$ we have
  \begin{equation*}
    \abs[0]{ I_{\alpha}(\phi_{\ell,2q}) - Q_{\alpha,\ell,n}^\textsm{GH}(\phi_{\ell,2q}) } \leq  C_q^{-1}  \frac{\beta}{\alpha} \bigg( \frac{\beta^2}{\ell^2} \bigg)^q q^{-1/4},
  \end{equation*}
  where $C_q > 0$ is defined in~\eqref{eq:C}.
\end{lemma}
\begin{proof} 
  For $q < n$ the statement is trivial since the scaled Gauss--Hermite rule is exact for $\phi_{\ell,2q}$. For $q \geq n$ write
  \begin{equation*}
    \abs[0]{I_{\alpha}(\phi_{\ell,2q}) - Q_{\alpha,\ell,n}^\textsm{GH}(\phi_{\ell,2q})} = \frac{\beta}{\alpha} \frac{1}{\ell^{2q} \sqrt{(2q)!}} \abs[0]{ I_{\beta}(x^{2q}) - Q_{\beta,n}^\textsm{GH}(x^{2q}) }.
  \end{equation*}
  By the positivity of the Gauss--Hermite weights and~\eqref{eq:GH-error} we have, for some $\xi \in \R$,
  \begin{equation*}
    0 < Q_{\beta,n}^\textsm{GH}(x^{2q}) = I_{\beta}(x^{2q}) - \frac{\beta^{2n} n!}{(2n!)} \frac{(2q)!}{(2(q-n))!} \xi^{2(q-n)} \leq I_{\beta}(x^{2q}).
  \end{equation*}
  Therefore $\abs[0]{I_\beta(x^{2q}) - Q_{\beta,n}^\textsm{GH}(x^{2q}) } \leq I_\beta(x^{2q})$.
  The triangle inequality and the Gaussian moment formula then yield
  \begin{equation*}
    \begin{split}
      \abs[0]{I_{\alpha}(\phi_{\ell,2q}) - Q_{\alpha,\ell,n}^\textsm{GH}(\phi_{\ell,2q})} \leq \frac{\beta}{\alpha} \frac{1}{\ell^{2q} \sqrt{(2q)!}} I_{\beta}(x^{2q}) &=  \frac{\beta}{\alpha}  \frac{1}{\ell^{2q} \sqrt{(2q)!}} \frac{ \beta^{2q} (2q)!}{2^q q!} \\
      &=  \frac{\sqrt{(2q)!}}{2^q q!}  \frac{\beta}{\alpha} \bigg(\frac{\beta^2}{\ell^2} \bigg)^q \\
      &= C_q^{-1} \frac{\beta}{\alpha} \bigg(\frac{\beta^2}{\ell^2} \bigg)^q q^{-1/4}.
      \end{split}
  \end{equation*}
\end{proof}

\begin{theorem} \label{thm:gh-1d} For any $n \geq 1$ we have
  \begin{equation} \label{eq:gh-1d}
    C_n \frac{\ell}{\sqrt{\alpha^2+\ell^2}} \bigg( \frac{\alpha^2}{2(\alpha^2 + \ell^2)} \bigg)^n n^{1/4} \leq e_{\alpha,\ell}(Q_{\alpha,\ell,n}^\textsm{GH}) <  \pi^{-1/4}  \frac{\ell}{\sqrt{\alpha^2+\ell^2}} \bigg( \frac{\alpha^2}{\alpha^2 + \ell^2} \bigg)^n n^{-1/4},
  \end{equation}
  where $C_n > 0$ is defined in~\eqref{eq:C}.
\end{theorem}
\begin{proof} 
  Since $\norm[0]{\phi_{\ell,m}}_{\ell} = 1$ for every $m \geq 0$, the lower bound follows immediately from Lemma~\ref{lemma:2n}.
  Let $f = \sum_{m=0}^\infty f_m \phi_{\ell,m} \in \rkhs(K_\ell)$. Because $Q_{\alpha,\ell,n}^\textsm{GH}$ is exact for $\phi_{\ell,0}, \ldots, \phi_{\ell,2n-1}$ and $Q_{\alpha,\ell,n}^\textsm{GH}(\phi_{\ell,m}) = I_{\alpha}(\phi_{\ell,m}) = 0$ if $m$ is odd, 
  \begin{equation*}
    \begin{split}
      \abs[0]{I_{\alpha}(f) - Q_{\alpha,\ell,n}^\textsm{GH}(f)} &\leq \sum_{m=0}^\infty \abs[0]{f_m} \abs[0]{ I_{\alpha}(\phi_{\ell,m}) - Q_{\alpha,\ell,n}^\textsm{GH}(\phi_{\ell,m})} \\
      &= \sum_{q=n}^\infty \abs[0]{f_{2q}} \abs[0]{ I_{\alpha}(\phi_{\ell,2q}) - Q_{\alpha,\ell,n}^\textsm{GH}(\phi_{\ell,2q})}.
      \end{split}
  \end{equation*}
  Recall from Lemma~\ref{lemma:C} that $C_q^{-1} < \pi^{-1/4}$.
  Lemma~\ref{lemma:>2n} thus yields
  \begin{equation*}
    \begin{split}
    \abs[0]{I_{\alpha}(f) - Q_{\alpha,\ell,n}^\textsm{GH}(f)} &\leq \frac{\beta}{\alpha}  \sum_{q=n}^\infty \abs[0]{f_{2q}} C_q^{-1} \bigg( \frac{\beta^2}{\ell^2} \bigg)^q q^{-1/4} \\
    &= \frac{\beta}{\alpha}  \bigg( \frac{\beta^2}{\ell^2} \bigg)^{n} \sum_{q=0}^\infty \abs[0]{f_{2(n+q)}} C_{n+q}^{-1} \bigg( \frac{\beta^2}{\ell^2} \bigg)^q (n+q)^{-1/4} \\
    &< \pi^{-1/4} \frac{\beta}{\alpha} \bigg( \frac{\beta^2}{\ell^2} \bigg)^{n} n^{-1/4} \sum_{q=0}^\infty \abs[0]{f_{2(n+q)}} \bigg( \frac{\beta^2}{\ell^2} \bigg)^q.
    \end{split}
  \end{equation*}
  We conclude that
  \begin{equation*}
    \begin{split}
    \sup_{ \norm[0]{f}_{\ell} \leq 1} \abs[0]{I_{\alpha}(f) - Q_{\alpha,\ell,n}^\textsm{GH}(f)} &<  \pi^{-1/4} \frac{\beta}{\alpha} \bigg( \frac{\beta^2}{\ell^2} \bigg)^{n} n^{-1/4} \sup_{ \norm[0]{f}_\ell \leq 1} \, \sum_{q=0}^\infty \abs[0]{f_{2(n+q)}} \bigg( \frac{\beta^2}{\ell^2} \bigg)^q \\
    &=  \pi^{-1/4}  \frac{\beta}{\alpha} \bigg( \frac{\beta^2}{\ell^2} \bigg)^{n} n^{-1/4}
    \end{split}
  \end{equation*}
  because $\beta^2 < \ell^2$ and the supremum is thus attained by $f = \phi_{\ell,2n}$, hich corresponds  to $f_{2(n+q)} = 1$ for $q=0$ and $f_{2(n+q)} = 0$ for $q > 0$.
\end{proof}

\begin{figure}[t!]
  \centering
  \includegraphics{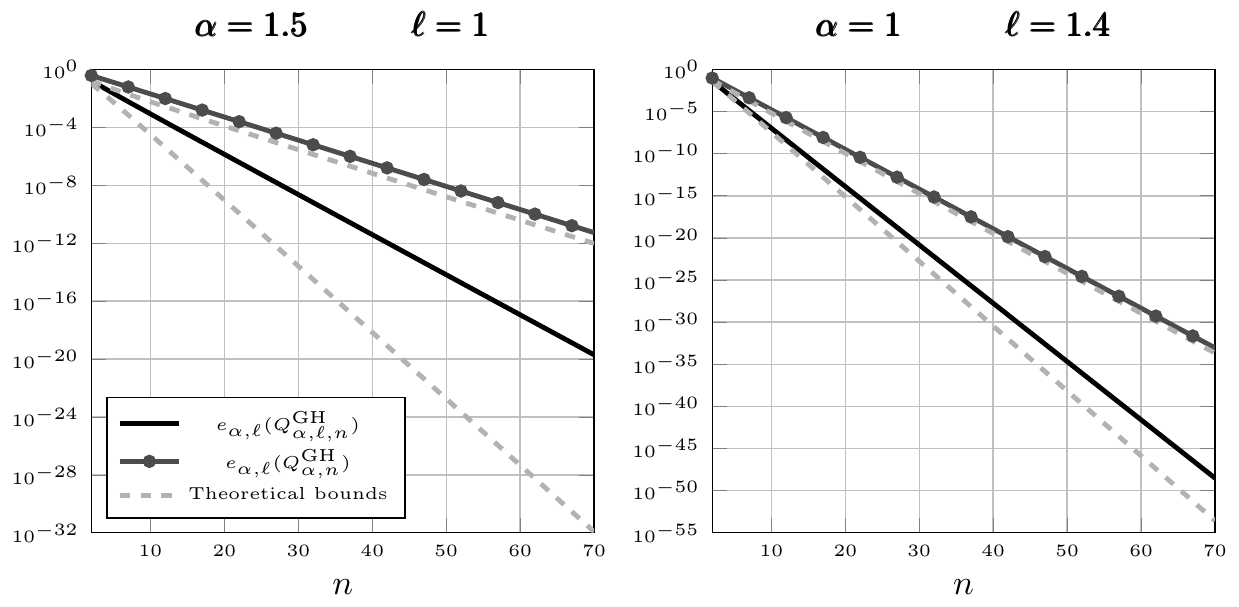}
  \caption{Worst-case errors of the scaled Gauss--Hermite rule $Q_{\alpha,\ell,n}^\textsm{GH}$ and the standard Gauss--Hermite rule $Q_{\alpha,n}^\textsm{GH}$ and our theoretical bounds~\eqref{eq:gh-1d} for $e_{\alpha,\ell}(Q_{\alpha,\ell,n}^\textsm{GH})$. 
  Note that the earlier results for the standard Gauss--Hermite rule in \citet{KuoWozniakowski2012} and \citet{KuoSloanWozniakowski2017} do not apply in the left panel because $\alpha > \ell$. 
  Rates of decay computed from the numerical results for scaled Gauss--Hermite rules are $r^n$ for $r\approx 0.528$ (left panel) and $r^n$ for $r\approx 0.203$ (right panel).
  Interestingly, the upper bound in~\eqref{eq:gh-1d} appears to be almost exact for the \emph{standard} Gauss--Hermite rule. All computations were implemented in Python with 100-digit precision.
} \label{fig:wce-gh-scaled}
\end{figure}

\begin{remark} \label{remark:gh-suboptimality} The exponential difference in the upper and lower bounds of Theorem~\ref{thm:gh-1d} partially stems from the rough estimate in Lemma~\ref{lemma:>2n}, which is merely double the integral of $\phi_{\ell,2q}$ (indeed, this estimate only depends on $q$, not on $n$). A more careful analysis of the Gauss--Hermite error for even polynomials, $\abs[0]{I_\beta(x^{2q}) - Q_{\beta,n}^\textsm{GH}(x^{2q})}$, could be expected to yield improvements. See Figure~\ref{fig:wce-gh-scaled} for numerical results.
\end{remark}

\citet{KuoWozniakowski2012} and \citet{KuoSloanWozniakowski2017} analyse integration in $\mathcal{H}(K_{\ell})$ under the constraint that $\ell > \alpha$.
Theorem~4.1 in \citet{KuoSloanWozniakowski2017} contains a lower bound for the $n$th minimal worst-case error
\begin{equation} \label{eq:nth-minimal}
  e_{\alpha,\ell,n}^\textsm{min} \coloneqq \inf_{ Q_n } e_{\alpha,\ell}(Q_n),
\end{equation}
where the infimum is taken over all $n$-point quadrature rules $Q_n$.
A careful reading reveals that the assumption $\ell > \alpha$ is not required in the proof of the lower bound.
Let
\begin{equation} \label{eq:gamma-omega}
  \gamma \coloneqq \frac{\alpha}{\ell} \quad \text{ and } \quad \omega_\gamma \coloneqq \frac{2\gamma^2}{1+2\gamma^2 + \sqrt{1+4\gamma^2}} < 1.
\end{equation}
Then a generalisation of Theorem~4.1 in \citet{KuoSloanWozniakowski2017} states that
\begin{equation} \label{eq:kuo-lower}
  e_{\alpha,\ell,n}^\textsm{min} \geq \sqrt{ \frac{2(1+4\gamma^2)^{1/4}}{(1+2\gamma^2+\sqrt{1+4\gamma^2})\e} } \, \frac{\omega_\gamma^n n!}{(n+1) (2n)!}.
\end{equation}
Because $\omega_\ell < 1$ and
\begin{equation} \label{eq:stirling-for-minimal}
  \frac{n!}{(2n)!} \sim \frac{\e^n}{2^{2n+1/2} n^{n} }, \quad \text{ which for finite $n$ is } \quad \frac{n!}{(2n)!} \geq \sqrt{\pi} \frac{ \e^{n-1}}{ 2^{2n} n^{n} },
\end{equation}
this lower bound is super-exponential and likely non-strict; see p.\@~847 in \citet{KuoSloanWozniakowski2017} for more discussion.
By combining the lower bound~\eqref{eq:kuo-lower}, the Stirling estimates~\eqref{eq:stirling-for-minimal} and the upper bound from Theorem~\ref{thm:gh-1d},  we obtain the first (at least) exponential bounds on the $n$th minimal error for all values of $\alpha$ and $\ell$ in this setting.

\begin{theorem} \label{thm:nth-minimal-1d}
  For any $n \geq 1$ the $n$th minimal error~\eqref{eq:nth-minimal} satisfies
  \begin{equation*}
    \bar{C}_{n}(\gamma) \bigg( \frac{\omega_\gamma \e}{4n} \bigg)^n (n+1)^{-1} \leq e_{\alpha,\ell,n}^\textsm{min} <  \pi^{-1/4}  \frac{\ell}{\sqrt{\alpha^2 + \ell^2}} \bigg( \frac{\alpha^2}{\alpha^2+\ell^2} \bigg)^n n^{-1/4},
  \end{equation*}
  where $\gamma$ and $\omega_\gamma$ are defined in~\eqref{eq:gamma-omega} and
  \begin{equation} \label{eq:C-bar-gamma}
    \bar{C}_n(\gamma) \coloneqq \sqrt{ \frac{2(1+4\gamma^2)^{1/4}}{(1+2\gamma^2+\sqrt{1+4\gamma^2})\e} } \, \frac{n!}{(2n)!} \bigg( \frac{\e}{4n} \bigg)^{-n}
  \end{equation}
  is a positive sequence such that
  \begin{equation*}
    \lim_{n \to \infty} \bar{C}_{n}(\gamma) = \sqrt{ \frac{(1+4\gamma^2)^{1/4}}{(1+2\gamma^2+\sqrt{1+4\gamma^2}\,)\e} } \quad \text{and} \quad \bar{C}_n(\gamma) \geq \sqrt{ \frac{2\pi(1+4\gamma^2)^{1/4}}{(1+2\gamma^2+\sqrt{1+4\gamma^2})\e^3} }.
  \end{equation*}  
\end{theorem}

\subsection{Error estimates for tensor product rules}

Let $\b{n} \coloneqq (n_1, \cdots, n_d) \in \N^d$. We now consider the tensor product extensions
\begin{equation} \label{eq:scaled-GH-tensor}
  Q_{\b{\alpha}, \b{\ell}, \b{n}}^\textsm{GH} \coloneqq Q_{\alpha_1, \ell_1, n_1}^\textsm{GH} \times \cdots \times Q_{\alpha_d, \ell_d, n_d}^\textsm{GH}
\end{equation}
of the scaled Gauss--Hermite rules defined in Section~\ref{sec:scaled-GH-subsec}. The approximation to $I_{\b{\alpha}}(f)$ for $f \colon \R^d \to \R$ is thus
\begin{equation*}
  Q_{\b{\alpha}, \b{\ell}, \b{n}}^\textsm{GH}(f) = \sum_{ \b{i} \in \N^d, \, \b{i} \leq \b{n} } w_{\b{\alpha}, \b{\ell}, \b{n}, \b{i}}^\textsm{GH} f( \b{x}_{\b{\alpha}, \b{\ell}, \b{n}, \b{i}}^\textsm{GH} ),
\end{equation*}
where $\b{i} \leq \b{n}$ stands for $i_j \leq n_j$ for every $j=1,\ldots,d$ and the points and weights are defined using the univariate versions in~\eqref{eq:scaled-GH-weights-point} as follows:
\begin{equation*}
  \b{x}_{\b{\alpha}, \b{\ell}, \b{n}, \b{i}}^\textsm{GH} \coloneqq (x_{\alpha_1, \ell_1, n_1, i_1}^\textsm{GH}, \ldots, x_{\alpha_d, \ell_d, n_d, i_d}^\textsm{GH} )  \quad \text{ and } \quad w_{\b{\alpha}, \b{\ell}, \b{n}, \b{i}}^\textsm{GH} \coloneqq \prod_{j=1}^d w_{\alpha_j, \ell_j, n_j, i_j}^\textsm{GH}.
\end{equation*}
Recall from Section~\ref{sec:introduction} that there exist representers $\mathcal{I}_{\b{\alpha},\b{\ell}}$ and $\mathcal{Q}_{\b{\alpha}, \b{\ell}, \b{n}}^\textsm{GH}$ in $\mathcal{H}(K_{\b{\ell}})$ such that
\begin{equation*}
  I_{\b{\alpha}}(f) = \inprod{f}{\mathcal{I}_{\b{\alpha},\b{\ell}}}_{\b{\ell}} \quad \text{ and } \quad Q_{\b{\alpha}, \b{\ell}, \b{n}}^\textsm{GH}(f) = \inprod{f}{\mathcal{Q}_{\b{\alpha}, \b{\ell}, \b{n}}^\textsm{GH}}_{\b{\ell}}
\end{equation*}
for every $f \in \mathcal{H}(K_{\b{\ell}})$ and that
\begin{equation} \label{eq:wce-representer}
  e_{\b{\alpha}, \b{\ell}}( Q_{\b{\alpha}, \b{\ell}, \b{n}}^\textsm{GH} ) = \norm[0]{ I_{\b{\alpha},\b{\ell}} - \mathcal{Q}_{\b{\alpha}, \b{\ell}, \b{n}}^\textsm{GH} }_{\b{\ell}}.
\end{equation}
Furthermore, the representers have the explicit forms
\begin{equation*}
  \begin{split}
    \mathcal{I}_{\b{\alpha},\b{\ell}}(\b{x}) = \prod_{i=1}^d \mathcal{I}_{\alpha_i, \ell_i}(x_i) = \prod_{i=1}^d \Bigg[ \frac{1}{\sqrt{2\pi} \alpha_i} \int_{\R} K_{\ell_i}(y_i, x_i) \exp\bigg(\! -\frac{y_i^2}{2\alpha_i^2} \bigg) \dif y_i \Bigg] 
    \end{split}
\end{equation*}
and
\begin{equation*}
  \begin{split}
    \mathcal{Q}_{\b{\alpha}, \b{\ell}, \b{n}}^\textsm{GH}(\b{x}) = \prod_{i=1}^d \mathcal{Q}_{\alpha_i, \ell_i, n_i}^\textsm{GH}(x_i) = \prod_{i=1}^d \Bigg[ \sum_{j=1}^{n_i} w_{\alpha_i,\ell_i,n_i,j}^\textsm{GH} K_{\ell_i}\big( x_{\alpha_i,\ell_i,n_i,j}^\textsm{GH}, x_i \big) \Bigg]
    \end{split}
\end{equation*}
for $\b{x} \in \R^d$.

\begin{lemma} \label{lemma:representers}
  For any $\alpha, \ell > 0$ and $n \geq 1$ we have
  \begin{equation*}
    \norm[0]{\mathcal{Q}_{\alpha, \ell, n}^\textsm{GH}}_{\ell} \leq \norm[0]{\mathcal{I}_{\alpha,\ell}}_{\ell} = \bigg( 1 + \frac{2\alpha^2}{\ell^2} \bigg)^{-1/4} < 1.
  \end{equation*}
\end{lemma}
\begin{proof} Recall from Section~\ref{sec:introduction} that $\norm[0]{\mathcal{I}_{\alpha,\ell}}_{\ell}^2 = I_{\alpha}(\mathcal{I}_{\alpha,\ell})$ and $\norm[0]{\mathcal{Q}_{\alpha,\ell,n}^\textsm{GH}}_{\ell}^2 = Q_{\alpha,\ell,n}^\textsm{GH}(\mathcal{Q}_{\alpha\ell,n}^\textsm{GH})$.
  It is then fairly straightforward to compute that
  \begin{equation} \label{eq:integral-representer-explicit}
    \mathcal{I}_{\alpha,\ell}(x) = \frac{1}{\sqrt{2\pi} \alpha} \int_\R K_\ell(y,x) \exp\bigg(\!-\frac{y^2}{2\alpha^2} \bigg) \dif y = \bigg( \frac{\ell^2}{\alpha^2 + \ell^2} \bigg)^{1/2} \exp\bigg(\!-\frac{x^2}{2(\alpha^2 + \ell^2)} \bigg)
  \end{equation}
  and
  \begin{equation*}
    \begin{split}
      \norm[0]{ \mathcal{I}_{\alpha,\ell} }_{\ell} &= \Bigg( \frac{1}{\sqrt{2\pi} \alpha} \int_\R \bigg( \frac{\ell^2}{\alpha^2 + \ell^2} \bigg)^{1/2} \exp\bigg(\!-\frac{x^2}{2(\alpha^2 + \ell^2)} \bigg)  \exp\bigg(\!-\frac{x^2}{2\alpha^2} \bigg) \dif x \Bigg)^{1/2} \\
      &= \bigg( 1 + \frac{2\alpha^2}{\ell^2} \bigg)^{-1/4}.
    \end{split}
  \end{equation*}
  The norm of the quadrature representer is
  \begin{equation*}
    \begin{split}
      \norm[0]{\mathcal{Q}_{\alpha,\ell,n}^\textsm{GH}}_{\ell} ={}& \Bigg( \sum_{i=1}^n \sum_{j=1}^n w_{\alpha,\ell,n,i}^\textsm{GH} w_{\alpha,\ell,n,j}^\textsm{GH} K_\ell \big(x_{\alpha,\ell,n,i}^\textsm{GH}, x_{\alpha,\ell,n,j}^\textsm{GH} \big) \Bigg)^{1/2} \\
      ={}& \frac{\beta}{\alpha} \Bigg( \sum_{i=1}^n \sum_{j=1}^n w_{n,i}^\textsm{GH} w_{n,j}^\textsm{GH} \exp\bigg( \frac{\beta^2 (x_{n,i}^\textsm{GH})^2}{2\ell^2} \bigg) \exp\bigg( \frac{\beta^2 (x_{n,j}^\textsm{GH})^2}{2\ell^2} \bigg)  \Bigg. \\
        &\hspace{2.5cm}\times \Bigg. \exp\bigg( \!- \frac{\beta^2(x_{n,i}^\textsm{GH}-x_{n,j}^\textsm{GH})^2}{2\ell^2}\bigg) \Bigg)^{1/2} \\
      ={}& \frac{\beta}{\alpha} \Bigg( \sum_{i=1}^n w_{n,i}^\textsm{GH} \sum_{j=1}^n w_{n,j}^\textsm{GH} \exp\bigg( \frac{\beta^2 x_{n,i}^\textsm{GH} x_{n,j}^\textsm{GH}}{\ell^2} \bigg) \Bigg)^{1/2}.
      \end{split}
  \end{equation*}
  We recognise the inner sum in the last equation as the Gauss--Hermite integral approximation $Q_{1,n}^\textsm{GH}(g_i)$ for the function $g_i(x) \coloneqq \exp( \beta^2 x_{n,i}^\textsm{GH} x/\ell^2)$.
  Because derivatives of all orders of these function are everywhere positive, we conclude from~\eqref{eq:GH-error} that
  \begin{equation*}
    \begin{split}
      \sum_{j=1}^n w_{n,j}^\textsm{GH} \exp\bigg( \frac{\beta^2 x_{n,i}^\textsm{GH} x_{n,j}^\textsm{GH}}{\ell^2} \bigg) &\leq \frac{1}{\sqrt{2\pi} } \int_\R \exp\bigg( \frac{\beta^2 x_{n,i}^\textsm{GH} x}{\ell^2}\bigg) \exp\bigg(\! - \frac{x^2}{2} \bigg) \dif x \\
      &= \exp\bigg( \frac{\beta^4 (x_{n,i}^\textsm{GH})^2}{2\ell^4} \bigg).
      \end{split}
  \end{equation*}
  The positivity of the Gauss--Hermite weights thus gives
  \begin{equation} \label{eq:Q-representer-intermediate}
    \norm[0]{\mathcal{Q}_{\alpha,\ell,n}^\textsm{GH}}_{\ell} \leq \frac{\beta}{\alpha} \Bigg( \sum_{i=1}^n w_{n,i}^\textsm{GH} \exp\bigg( \frac{\beta^4 (x_{n,i}^\textsm{GH})^2}{2\ell^4} \bigg) \Bigg)^{1/2},
  \end{equation}
  where the sum is the Gauss--Hermite approximation of $I_1(g)$ for $g(x) \coloneqq \exp( \beta^4 x^2/(2\ell^4) )$.
  Because even order derivatives of $g$ are everywhere positive, it follows from~\eqref{eq:GH-error} that
  \begin{equation*}
    \begin{split}
    \sum_{i=1}^n w_{n,i}^\textsm{GH} \exp\bigg( \frac{\beta^4 (x_{n,i}^\textsm{GH})^2}{2\ell^4} \bigg) &\leq \frac{1}{\sqrt{2\pi}} \int_\R \exp\bigg( \frac{\beta^4 x^2}{2\ell^4}\bigg) \exp\bigg(\! - \frac{x^2}{2} \bigg) \dif x \\
    &= \frac{\ell^2}{\sqrt{\ell^4-\beta^4}} \\
    &= \bigg(1 - \frac{\alpha^4}{(\alpha^2+\ell^2)^2} \bigg)^{-1/2}.
    \end{split}
  \end{equation*}
  Inserting the above estimate into~\eqref{eq:Q-representer-intermediate} and observing that
  \begin{equation*}
    \frac{\beta}{\alpha} \bigg( 1 - \frac{\alpha^4}{(\alpha^2+\ell^2)^2} \bigg)^{-1/4} = \frac{\beta}{\alpha}\bigg(1 - \frac{\beta^4}{\ell^4} \bigg)^{-1/4} = \bigg( \frac{\alpha^4}{\beta^4} - \frac{\alpha^4}{\ell^4} \bigg)^{-1/4} = \bigg( 1 + \frac{2\alpha^2}{\ell^2} \bigg)^{-1/4}
  \end{equation*}
  yields the claim.
\end{proof}

\begin{lemma} \label{lemma:lower-bound-d} Let $\b{n} \in \N^d$ and $1 \leq i \leq d$. For the function
  \begin{equation*}
    f_i(\b{x}) \coloneqq \phi_{\ell_1,0}(x_1) \cdots \phi_{\ell_{i-1},0}(x_{i-1}) \; \phi_{\ell_i, 2n_i}(x_i) \; \phi_{\ell_{i+1},0}(x_{i+1}) \cdots \phi_{\ell_d,0}(x_d)
  \end{equation*}
  we have
  \begin{equation*}
    I_{\b{\alpha}}(f_i) - Q_{\b{\alpha}, \b{\ell}, \b{n}}^\textsm{GH}(f_i)  = C_{n_i} \Bigg[ \prod_{j = 1}^d \frac{\beta_j}{\alpha_j} \Bigg] \bigg( \frac{\beta_i^2}{2\ell_i^2} \bigg)^{n_i} n_i^{1/4},
  \end{equation*}
  where $C_{n_i} > 0$ is defined in~\eqref{eq:C} and 
  \begin{equation*}
    \beta_i \coloneqq \sqrt{ \frac{\alpha_i^2\ell_i^2}{\alpha_i^2+\ell_i^2} }.
  \end{equation*}
\end{lemma}
\begin{proof} Write
  \begin{equation*}
    \begin{split}
      I_{\b{\alpha}}(f_i) - Q_{\b{\alpha}, \b{\ell}, \b{n}}^\textsm{GH}(f_i) ={}&  \big[ I_{\alpha_i}(\phi_{\ell_i, 2n_i}) - Q_{\alpha_i,\ell_i,n_i}^\textsm{GH}(\phi_{\ell_i, 2n_i}) \big] \prod_{j \neq i} I_{\alpha_j}(\phi_{\ell_j,0}) \\
        &+ Q_{\alpha_i, \ell_i, n_i}^\textsm{GH}(\phi_{\ell_i, 2n_i}) \Bigg[\prod_{j \neq i} I_{\alpha_j}(\phi_{\ell_j,0}) - \prod_{j \neq i} Q_{\alpha_j, \ell_j, n_j}^\textsm{GH}(\phi_{\ell_j, 0})\Bigg].
        \end{split}
  \end{equation*}
  Since $Q_{\alpha_j, \ell_j, n_j}^\textsm{GH}(\phi_{\ell_j, 0}) = I_{\alpha_j}(\phi_{\ell_j,0})$ for $j=1,\ldots,d$, the second term vanishes.
  Because
  \begin{equation*}
    \prod_{j \neq i} I_{\alpha_j}(\phi_{\ell_j,0}) = \prod_{j \neq i} \Bigg[ \frac{1}{\sqrt{2\pi} \alpha_j} \int_{\R} \exp\bigg(\!-\frac{x_j^2}{2\ell_j^2} \bigg) \exp\bigg(\!-\frac{x_j^2}{2\alpha_j^2} \bigg) \dif x_j \Bigg] = \prod_{j \neq i} \frac{\beta_j}{\alpha_j},
  \end{equation*}
  the claim then follows from Lemma~\ref{lemma:2n}.
\end{proof}

\begin{theorem} \label{thm:gh-tensor} For any $\b{n} \in \N^d$, the tensor product rule~\eqref{eq:scaled-GH-tensor} satisfies
  \begin{equation*}
    e_{\b{\alpha}, \b{\ell}}( Q_{\b{\alpha}, \b{\ell}, \b{n}}^\textsm{GH} ) <  \pi^{-1/4}  \sum_{i=1}^d \frac{\ell_i}{(\alpha_i^2+\ell_i^2)^{1/2}} \Bigg[ \prod_{j \neq i} \bigg(1 + \frac{2\alpha_j^2}{\ell_j^2} \bigg)^{-1/4} \Bigg] \bigg( \frac{\alpha_i^2}{\alpha_i^2 + \ell_i^2} \bigg)^{n_i} n_i^{-1/4}
  \end{equation*}
  and
  \begin{equation*}
    e_{\b{\alpha}, \b{\ell}}( Q_{\b{\alpha}, \b{\ell}, \b{n}}^\textsm{GH} ) \geq \Bigg[ \prod_{j = 1}^d \frac{\ell_j}{(\alpha_j^2+\ell_j^2)^{1/2}} \Bigg] \min_{i=1,\ldots,d} C_{n_i} \bigg( \frac{\alpha_i^2}{2(\alpha_i^2+\ell_i^2)} \bigg)^{n_i} n_i^{1/4} ,
  \end{equation*}
  where $C_{n_i} > 0$ is defined in~\eqref{eq:C}.
\end{theorem}
\begin{proof}
  The lower bound follows directly from Lemma~\ref{lemma:lower-bound-d} because for each $i = 1,\ldots,d$ the function $f_i$ is one of the orthonormal basis functions~\eqref{eq:phi-d} and thus of unit norm.
  The proof of the upper bound is fairly standard.
  We use the representer form of the worst-case error in~\eqref{eq:wce-representer}.
  For any $\b{z} \in \R^d$ and $2 \leq q \leq d+1$ define $\b{z}_{1:q} \coloneqq (z_1, \ldots, z_{q-1}) \in \R^{q-1}$.
  Also denote
  \begin{equation*}
    \mathcal{I}_{\b{\alpha},\b{\ell}}[q] \coloneqq \mathcal{I}_{\b{\alpha}_{1:q}, \b{\ell}_{1:q}}, \quad \mathcal{Q}_{\b{\alpha}, \b{\ell}, \b{n}}^\textsm{GH}[q] \coloneqq \mathcal{Q}_{\b{\alpha}_{1:q}, \b{\ell}_{1:q}, \b{n}_{1:q}}^\textsm{GH} \quad \text{ and } \quad Q_{\b{\alpha}, \b{\ell}, \b{n}}^\textsm{GH}[q] \coloneqq Q_{\b{\alpha}_{1:q}, \b{\ell}_{1:q}, \b{n}_{1:q}}^\textsm{GH}.
  \end{equation*}
  Write
  \begin{equation*}
      \mathcal{I}_{\b{\alpha},\b{\ell}} - \mathcal{Q}_{\b{\alpha}, \b{\ell}, \b{n}}^\textsm{GH} = \mathcal{I}_{\b{\alpha},\b{\ell}}[d] \big( \mathcal{I}_{\alpha_d, \ell_d} - \mathcal{Q}_{\alpha_d,\ell_d,n_d}^\textsm{GH} \big) + \mathcal{Q}_{\alpha_d,\ell_d,n_d}^\textsm{GH} \big( \mathcal{I}_{\b{\alpha},\b{\ell}}[d] - \mathcal{Q}_{\b{\alpha}, \b{\ell}, \b{n}}^\textsm{GH}[d] \big).
  \end{equation*}
  Therefore
  \begin{equation*}
    \begin{split}
      e_{\b{\alpha}, \b{\ell}}( Q_{\b{\alpha}, \b{\ell}, \b{n}}^\textsm{GH} ) = \norm[0]{ \mathcal{I}_{\b{\alpha},\b{\ell}} - \mathcal{Q}_{\b{\alpha}, \b{\ell}, \b{n}}^\textsm{GH} }_{\b{\ell}}
      \leq{}& \norm[0]{\mathcal{I}_{\b{\alpha},\b{\ell}}[d]}_{\b{\ell}_{1:d}} e_{\alpha_d,\ell_d}( Q_{\alpha_d,\ell_d,n_d}^\textsm{GH} ) \\
      &+ \norm[0]{ \mathcal{Q}_{\alpha_d,\ell_d,n_d}^\textsm{GH} }_{\ell_d} e_{\alpha_{1:d}, \ell_{1:d}}\big( Q_{\b{\alpha}, \b{\ell}, \b{n}}^\textsm{GH}[d] \big).
      \end{split}
  \end{equation*}
  Iteration of this inequality and repeated applications of Lemma~\ref{lemma:representers} yield
\begin{equation*}
  \begin{split}
    e_{\b{\alpha}, \b{\ell}}( Q_{\b{\alpha}, \b{\ell}, \b{n}}^\textsm{GH} ) \leq{}& e_{\alpha_d,\ell_d}( Q_{\alpha_d,\ell_d,n_d}^\textsm{GH} ) \prod_{j=1}^{d-1} \bigg(1 + \frac{2\alpha_j^2}{\ell_j^2} \bigg)^{-1/4} + \bigg(1 + \frac{2\alpha_d^2}{\ell_d^2} \bigg)^{-1/4} e_{\alpha_{1:d}, \ell_{1:d}}\big( Q_{\b{\alpha}, \b{\ell}, \b{n}}^\textsm{GH}[d] \big) \\
    \leq{}& e_{\alpha_d,\ell_d}( Q_{\alpha_d,\ell_d,n_d}^\textsm{GH} ) \prod_{j=1}^{d-1} \bigg(1 + \frac{2\alpha_j^2}{\ell_j^2} \bigg)^{-1/4} \\
    &+ \bigg(1 + \frac{2\alpha_d^2}{\ell_d^2} \bigg)^{-1/4} \Bigg[ e_{\alpha_{d-1},\ell_{d-1}}( Q_{\alpha_{d-1},\ell_{d-1},n_{d-1}}^\textsm{GH} ) \prod_{j=1}^{d-2} \bigg(1 + \frac{2\alpha_j^2}{\ell_j^2} \bigg)^{-1/4} \Bigg. \\
      &\hspace{3.1cm} \Bigg.+ \bigg(1 + \frac{2\alpha_{d-1}^2}{\ell_{d-1}^2} \bigg)^{-1/4} e_{\alpha_{1:d-1}, \ell_{1:d-1}}\big( Q_{\b{\alpha}, \b{\ell}, \b{n}}^\textsm{GH}[d-1] \big) \Bigg] \\
    \vdots{}& \\
    \leq{}& \sum_{i=1}^d e_{\alpha_i,\ell_i}( Q_{\alpha_i,\ell_i,n_i}^\textsm{GH} ) \prod_{j \neq i} \bigg(1 + \frac{2\alpha_j^2}{\ell_j^2} \bigg)^{-1/4}.
    \end{split}
\end{equation*}
The claim then follows from the upper bound in~\eqref{eq:gh-1d} applied to each of the $d$ one-dimensional worst-case errors.
\end{proof}

In the isotropic case the statement of Theorem~\ref{thm:gh-tensor} simplifies considerably.

\begin{corollary} \label{cor:gh-tensor-isotropic} Consider the tensor product rule~\eqref{eq:scaled-GH-tensor} when $\alpha_1 = \cdots = \alpha_d = \alpha$, $\ell_1 = \cdots = \ell_d = \ell$ and \sloppy{${n_1 = \cdots = n_d = n}$} for $\alpha,\ell>0$ and $n \geq 1$. Then
  \begin{equation*}
    e_{\b{\alpha}, \b{\ell}}(Q_{\b{\alpha}, \b{\ell}, \b{n}}^\textsm{GH} ) < d\pi^{-1/4}  \frac{\ell}{\sqrt{\alpha^2+\ell^2}} \bigg(1+\frac{2\alpha^2}{\ell^2} \bigg)^{-(d-1)/4} \bigg(\frac{\alpha^2}{\alpha^2+\ell^2}\bigg)^n n^{-1/4}
  \end{equation*}
  and
  \begin{equation*}
    e_{\b{\alpha}, \b{\ell}}(Q_{\b{\alpha}, \b{\ell}, \b{n}}^\textsm{GH} ) \geq C_n \bigg( \frac{\ell^2}{\alpha^2+\ell^2}\bigg)^{d/2} \bigg(\frac{\alpha^2}{2(\alpha^2+\ell^2)}\bigg)^n n^{1/4},
  \end{equation*}
  where $C_n > 0$ is defined in~\eqref{eq:C}.
\end{corollary}

\begin{remark}
As the total number of points in Corollary~\ref{cor:gh-tensor-isotropic} is $N = n^d$, we obtain
\begin{equation} \label{eq:Nd-GH}
  C_1 \bigg( \frac{\alpha^2}{2(\alpha^2 + \ell^2)} \bigg)^{N^{1/d}} N^{1/(4d)} \leq e_{\b{\alpha}, \b{\ell}}(Q_{\b{\alpha}, \b{\ell}, \b{n}}^\textsm{GH} ) < C_2 \bigg( \frac{\alpha^2}{\alpha^2 + \ell^2} \bigg)^{N^{1/d}} N^{-1/(4d)}
\end{equation}
for certain constants $C_1, C_2 > 0$.
The curse of dimensionality thus manifests itself in the exponent $N^{1/d}$ that grows slower with $N$ when $d$ is large.
From~\eqref{eq:Nd-GH} one could derive a number of dimensional tractability results, as is done for tensor products of Gauss--Hermite rules in \citet{KuoSloanWozniakowski2017}.
\end{remark}

As a final result of this section we provide a multivariate generalisation of Theorem~\ref{thm:nth-minimal-1d}.
Let $N+1 = \prod_{i=1}^d(n_i+1)$ for any $n_i \geq 1$. As in the one-dimensional case, \citet[Theorem~5.1]{KuoSloanWozniakowski2017} have proved the lower bound
\begin{equation} \label{eq:min-error-lower-multi}
  e_{\b{\alpha}, \b{\ell}, N}^\textsm{min} \geq \frac{1}{N+1} \prod_{i=1}^d \Bigg[ \sqrt{ \frac{2(1+4\gamma_i^2)^{1/4} }{ (1+2\gamma_i^2+(1+4\gamma_i^2)^{1/2})\e  } } \, \frac{\omega_{\gamma_i}^{n_i} n_i!}{(2n_i)!} \Bigg]
\end{equation}
for the $N$th minimal error
\begin{equation} \label{eq:nth-minimal-d}
  e_{\b{\alpha}, \b{\ell}, N}^\textsm{min} \coloneqq \inf_{ Q_N } e_{\b{\alpha}, \b{\ell}}( Q_N ),
\end{equation}
where the infimum is over $d$-dimensional $N$-point quadrature rules $Q_N$.
Here
\begin{equation} \label{eq:gamma-omega-d}
  \gamma_i \coloneqq \frac{\alpha_i}{\ell_i} \quad \text{ and } \quad \omega_{\gamma_i} \coloneqq \frac{2\gamma_i^2}{1+2\gamma_i^2 + \sqrt{1+4\gamma_i^2}} < 1.
\end{equation}
Combining the upper bound of Theorem~\ref{thm:gh-tensor} and~\eqref{eq:stirling-for-minimal} with the bound~\eqref{eq:min-error-lower-multi} yields the following result, where the upper bound based on the tensor product rule with $\prod_{i=1}^d n_i < N$ points is valid since the minimal error is decreasing in the number of points.

\begin{theorem} \label{thm:nth-minimal-d} For any $N \geq 1$ such that $N+1=\prod_{i=1}^d(n_i+1)$ for some $n_i \geq 1$ the $N$th minimal error~\eqref{eq:nth-minimal-d} satisfies
  \begin{equation*}
    (N+1)^{-1} \prod_{i=1}^d \bar{C}_{n_i}(\gamma_i) \bigg( \frac{\omega_{\gamma_i} \e}{4n_i} \bigg)^{n_i} \leq e_{\b{\alpha}, \b{\ell}, N}^\textsm{min} < \pi^{-1/4}  \sum_{i=1}^d \widehat{C}_i \bigg( \frac{\alpha_i^2}{\alpha_i^2 + \ell_i^2} \bigg)^{n_i} n_i^{-1/4},
  \end{equation*}
  where $\gamma_i$ and $\omega_{\gamma_i}$ are defined in~\eqref{eq:gamma-omega-d}, $(\bar{C}_n(\gamma))_{n=1}^\infty$ is the positive sequence in~\eqref{eq:C-bar-gamma} and
  \begin{equation*}
    \widehat{C}_i \coloneqq \frac{\ell_i}{(\alpha_i^2+\ell_i^2)^{1/2}} \prod_{j \neq i} \bigg(1 + \frac{2\alpha_j^2}{\ell_j^2} \bigg)^{-1/4}.
  \end{equation*}
\end{theorem}

\section{Locally uniform points and optimal weights} \label{sec:kernel-quadrature}

This section contains a flexible construction which permits nested point sets in situations where an extensible integration rule is required.
In contrast to the scaled Gauss--Hermite rules in Section~\ref{sec:scaled-GH}, this construction is only proved to converge with a sub-exponential (though still super-algebraic) rate.
The construction and its analysis are based on results in scattered data approximation literature~\citep{Wendland2005,FasshauerMcCourt2015} and worst-case optimal integration rules in RKHSs~\citep{Oettershagen2017}.

\subsection{Rules with optimal weights}

Let $X = \{\b{x}_1, \ldots, \b{x}_n\} \subset \R^d$ be an arbitrary set of $n$ distinct points.
The integration rule based on these points having the minimal worst-case error is
\begin{equation*}
  Q_{\b{\alpha}, \b{\ell}, X}^\textsm{opt}(f) \coloneqq \sum_{i=1}^n w_{\b{\alpha}, \b{\ell}, X, i}^\textsm{opt} f(\b{x}_i)
\end{equation*}
with the weights
\begin{equation*}
  (w_{\b{\alpha}, \b{\ell}, X, 1}^\textsm{opt}, \ldots, w_{\b{\alpha}, \b{\ell}, X, n}^\textsm{opt}) = \argmin_{ \b{w} \in \R^n } e_{\b{\alpha}, \b{\ell}}( Q_{X, \b{w}} ),
\end{equation*}
where $Q_{X, \b{w}}(f) \coloneqq \sum_{i=1}^n w_i f(\b{x}_i)$.
Because the explicit form of the worst-case error in~\eqref{eq:wce-explicit} is
\begin{equation} \label{eq:wce-linear-algebra}
  e_{\b{\alpha}, \b{\ell}}(Q_n) = \sqrt{ I_{\b{\alpha}}^{\b{x}} I_{\b{\alpha}}^{\b{y}}( K_{\b{\ell}}(\b{x}, \b{y})) - 2 \b{w}^\T \b{z} + \b{w}^\T \b{K}_X \b{w} },
\end{equation}
where $z_i = I_{\b{\alpha}}( K_{\b{\ell}}(\cdot, \b{x}_i)) = \mathcal{I}_{\b{\alpha}, \b{\ell}}(\b{x}_i)$ and $\b{K}_X$ is the $n \times n$ positive-definite kernel Gram matrix with elements $(\b{K}_X)_{i,j} = K_{\b{\ell}}(\b{x}_i, \b{x}_j)$, it is easy to see that the optimal weights are the solution to the linear system\footnote{The representers on the right-hand side can be computed in closed from by taking products of the one-dimensional representers in~\eqref{eq:integral-representer-explicit}. Because the Gaussian kernel is analytic, the linear system tends to become severely ill-conditioned~\citep{Schaback1995}, which can be somewhat mitigated by the use of approximations based on truncation of an orthonormal expansion of the Gaussian kernel~\citep{FasshauerMcCourt2012,KarvonenSarkka2019}.}
\begin{equation} \label{eq:linear-system}
  \begin{bmatrix} K_{\b{\ell}}(\b{x}_1, \b{x}_1) & \cdots & K_{\b{\ell}}(\b{x}_1, \b{x}_n) \\ \vdots & \ddots & \vdots \\ K_{\b{\ell}}(\b{x}_n, \b{x}_1) & \cdots & K_{\b{\ell}}(\b{x}_n, \b{x}_n) \end{bmatrix} \begin{bmatrix} w_{\b{\alpha}, \b{\ell}, X, 1}^\textsm{opt} \\ \vdots \\ w_{\b{\alpha}, \b{\ell}, X, n}^\textsm{opt} \end{bmatrix}
  = \begin{bmatrix} \mathcal{I}_{\b{\alpha}, \b{\ell}}(\b{x}_1) \\ \vdots \\ \mathcal{I}_{\b{\alpha}, \b{\ell}}(\b{x}_n) \end{bmatrix}.
\end{equation}
These integration rules, sometimes known as kernel quadrature rules, are useful because no restrictions are placed on the geometry of the evaluation points.
They also carry an interpretation as Bayesian quadrature rules~\citep{Briol2019} which can be used to quantify the epistemic uncertainty in the integral approximation.

Because the optimal weights solve~\eqref{eq:linear-system} it follows from~\eqref{eq:wce-linear-algebra} that
\begin{equation} \label{eq:wce-kq}
  \begin{split}
    e_{\b{\alpha}, \b{\ell}}(Q_{\b{\alpha}, \b{\ell}, X}^\textsm{opt}) = \norm[0]{ \mathcal{I}_{\b{\alpha}, \b{\ell}} - \mathcal{Q}_{\b{\alpha}, \b{\ell}, X}^\textsm{opt}}_{\b{\ell}} &= \sqrt{ \norm[0]{ \mathcal{I}_{\b{\alpha}, \b{\ell}} }_{\b{\ell}}^2 - \norm[0]{ \mathcal{Q}_{\b{\alpha}, \b{\ell}, X}^\textsm{opt}}_{\b{\ell}}^2 }\\
    &= \sqrt{ I_{\b{\alpha}}^{\b{x}} I_{\b{\alpha}}^{\b{y}}( K_{\b{\ell}}(\b{x}, \b{y})) - \sum_{i=1}^n w_{\b{\alpha}, \b{\ell}, X, i}^\textsm{opt} I_{\b{\alpha}}(K_{\b{\ell}}(\cdot, \b{x}_i))},
    \end{split}
\end{equation}
where $\mathcal{Q}_{\b{\alpha}, \b{\ell}, X}^\textsm{opt}$ is the representer of the integration rule $Q_{\b{\alpha}, \b{\ell}, X}^\textsm{opt}$.
To bound the worst-case error we use the connection between kernel quadrature rules and kernel interpolation. The kernel interpolant is the minimum-norm interpolant
\begin{equation} \label{eq:minimum-norm}
  s_{\b{\ell},X}f \coloneqq \argmin_{g \in \mathcal{H}(K_{\b{\ell}})} \Set{ \norm[0]{g}_{\b{\ell}} }{ g(\b{x}_i) = f(\b{x}_i) \text{ for all } \b{x}_i \in X}
\end{equation}
to $f$ at the points $X$ and it can be shown that the optimal integration rule is obtained by integrating this interpolant: $Q_{\b{\alpha}, \b{\ell}, X}^\textsm{opt}(f) = I_{\b{\alpha}}( s_{\b{\ell}, X} f)$.
The power function $P_{\b{\ell},X}$ is defined as the pointwise worst-case error of the kernel interpolant,
\begin{equation}\label{eq:power-function-wce}
  \begin{split}
  P_{\b{\ell},X}(\b{x}) &\coloneqq \sup_{ \norm[0]{f}_{\b{\ell}} \leq 1} \abs[0]{ f(\b{x}) - (s_{\b{\ell},X}f)(\b{x}) } \\
  &= \sup\Set{ f(\b{x}) }{ \norm[0]{f}_{\b{\ell}} \leq 1 \text{ and } f(\b{x}_i) = 0 \text{ for all } \b{x}_i \in X}.
  \end{split}
\end{equation}
The power function provides an error decoupling for approximation similar to~\eqref{eq:error-decoupling}:
\begin{equation} \label{eq:P-decoupling}
  \abs[0]{f(\b{x}) - (s_{\b{\ell}, X}f)(\b{x}) } \leq \norm[0]{f}_{\b{\ell}} P_{\b{\ell}, X}(\b{x})
\end{equation}
for any $f \in \mathcal{H}(K_{\b{\ell}})$ and $\b{x} \in \R^d$.
Since
\begin{equation} \label{eq:f-pointwise-RKHS}
  f(\b{x}) = \inprod{f}{K_{\b{\ell}}(\cdot, \b{x})}_{\b{\ell}} \leq \norm[0]{f}_{\b{\ell}} \norm[0]{K_{\b{\ell}}(\cdot, \b{x})}_{\b{\ell}} = \norm[0]{f}_{\b{\ell}} \sqrt{K_{\b{\ell}}(\b{x}, \b{x})} = \norm[0]{f}_{\b{\ell}}
\end{equation}
for $f \in \mathcal{H}(K_{\b{\ell}})$ and $\b{x} \in \R^d$, it follows from~\eqref{eq:power-function-wce} that $P_{\b{\ell}, X} \leq 1$.
Now, using~\eqref{eq:P-decoupling} the worst-case error can be bounded as follows:
\begin{equation} \label{eq:int-power-function}
  e_{\b{\alpha}, \b{\ell}}(Q_{\b{\alpha}, \b{\ell}, X}^\textsm{opt}) = \sup_{\norm[0]{f}_{\b{\ell}} \leq 1} \abs[0]{ I_{\b{\alpha}}(f) - Q_{\b{\alpha}, \b{\ell}, X}^\textsm{opt}(f) } \leq \sup_{\norm[0]{f}_{\b{\ell}} \leq 1} I_{\b{\alpha}}( \, \abs[0]{f - s_{\b{\ell},X}f } \,) \leq I_{\b{\alpha}}( P_{\b{\ell}, X } ).
\end{equation}

\subsection{Error estimates in one dimension}

We begin by presenting a general result on the $L^1$-norm of the power function on bounded cubes in dimension $d$.
For any finite point set $X = \{\b{x}_1, \ldots, \b{x}_n\} \subset \R^d$ define the fill-distance $h_{X,\Omega}$ on a bounded set $\Omega \subset \R^d$ as
\begin{equation} \label{eq:fill-distance}
  h_{X,\Omega} \coloneqq \sup_{ \b{x} \in \Omega} \, \min_{\b{x}_i \in X \cap \Omega} \, \norm[0]{ \b{x} - \b{x}_i}.
\end{equation}
Note that this differs from the standard definition of the fill-distance in scattered data approximation literature~\citep[e.g.,][Definition~1.4]{Wendland2005} in that $X$ is not required to be a subset of $\Omega$.
The following result is a localised version of the convergence results in \citet[Chapter~11]{Wendland2005} and \citet{RiegerZwicknagl2010}.
For other similar results, see \citet{RiegerZwicknagl2014}.
We use $\norm[0]{\cdot}_{L^1(\Omega)}$ to denote the $L^1$-norm on a Lebesgue-measurable set $\Omega \subset \R^d$. That is, $\norm[0]{f}_{L^1(\Omega)} = \int_\Omega \abs[0]{f(\b{x})} \dif \b{x}$.

\begin{proposition} \label{prop:rieger} Let $\Omega \subset \R^d$ be a closed cube with side length $R > 0$ and let $X \subset \R^d$ be a finite collection of distinct points.
  Consider the isotropic case $\b{\ell} = (\ell, \ldots, \ell) \in \R^d$ for some $\ell > 0$.
  Then there exist positive constants $C$ and $h_0$, which depend only on $\ell$, $d$ and $R$, such that
  \begin{equation} \label{eq:rieger}
    \norm[0]{ P_{\b{\ell},X} }_{L^1(\Omega)} \leq \exp\big( C \log ( h_{X,\Omega} ) \, h_{X,\Omega}^{-1} \, \big)
  \end{equation}
  whenever $h_{X,\Omega} \leq h_0$.
\end{proposition}
\begin{proof} 
  Let $Y \subset \Omega$ be a finite point set and $u \in \mathcal{H}(K_{\b{\ell}})$ a function that vanishes on $Y$. 
  Let $\norm[0]{\cdot}_{\b{\ell},\Omega}$ denote the norm of the restriction of $\mathcal{H}(K_{\b{\ell}})$ on $\Omega$. 
  By Theorems~4.5 and~6.1 in \citet{RiegerZwicknagl2010} with $p=2$ and $q=1$, there are positive constants $C$ and $h_0$, which depend only on $\ell$, $d$ and $R$, such that
  \begin{equation} \label{eq:sampling-inequality}
    \norm[0]{u}_{L^1(\Omega)} \leq \norm[0]{u}_{\b{\ell},\Omega} \exp\big( C \log (h_{Y,\Omega}) \, h_{Y,\Omega}^{-1} \, \big) \leq \norm[0]{u}_{\b{\ell}} \exp\big( C \log (h_{Y,\Omega}) \, h_{Y,\Omega}^{-1} \, \big) 
  \end{equation}
  if $h_{Y,\Omega} \leq h_0$.
  From the characterisation~\eqref{eq:power-function-wce} of the power function it then follows that
  \begin{equation*}
    \norm[0]{P_{\b{\ell},Y}}_{L^1(\Omega)} \leq \exp\big( C \log (h_{Y,\Omega}) h_{Y,\Omega}^{-1} \, \big).
  \end{equation*}
  and $P_{\b{\ell},X}(\b{x}) \leq P_{\b{\ell},Y}(\b{x})$ for every $\b{x} \in \R^d$ if $Y \subset X$. 
  Therefore
  \begin{equation*}
    \norm[0]{P_{\b{\ell},X}}_{L^1(\Omega)} \leq \norm[0]{P_{\b{\ell},X\cap\Omega}}_{L^1(\Omega)} \leq \exp\big( C \log (h_{X\cap\Omega,\Omega}) h_{X\cap\Omega,\Omega}^{-1} \, \big) = \exp\big( C \log (h_{X,\Omega}) h_{X,\Omega}^{-1} \, \big).
  \end{equation*}
\end{proof}

Next we consider the univariate case and apply Proposition~\ref{prop:rieger} after decomposing the full integration domain $\R$ into a number of disjoint unit intervals and a ``tail domain'' of the form $(-\infty, -a) \cup (a, \infty)$. The full one-dimensional Gaussian integral $I_\alpha(P_{\ell,X})$, which according to~\eqref{eq:int-power-function} is an upper bound to the worst-case error, is then evaluated by summing and appropriately weighting by the Gaussian weight function the $L^1$-norms of the power function on the intervals. If the points are selected in a suitable way, the resulting sum can be explicitly bounded. Section~\ref{sec:kernel-quad-multi} contains extensions for tensor product rules.
There are two principal reasons for using tensor products instead of constructing higher dimensional point sets and applying Proposition~\ref{prop:rieger} directly on them: (i) Proposition~\ref{prop:rieger} is available only for isotropic Gaussian kernels and (ii) in a multivariate version of~\eqref{eq:quasi-uniformity} the constant $c_\textsm{qu}$ in~\eqref{eq:quasi-uniformity} can no longer be independent of $m$ because, unlike in one dimension, the volume of a fixed width annulus depends on its radius.
The structure of a specific point set satisfying the assumptions of Proposition~\ref{prop:kernel-univariate-generic} and Theorem~\ref{thm:kq-univariate} can be seen in Figure~\ref{fig:vdc-2d} which depicts a product grid version.

\begin{proposition} \label{prop:kernel-univariate-generic} Let $(\bar{n}_m)_{m=1}^\infty$ be a strictly increasing sequence of positive integers and $(Y_m)_{m=1}^\infty$ a sequence of sets such that each $Y_m \subset (0,1)$ consists of $\bar{n}_m$ distinct points and the quasi-uniformity condition 
  \begin{equation} \label{eq:quasi-uniformity}
    h_{Y_m,(0,1)} \leq c_\textsm{qu} \bar{n}_m^{-1}
  \end{equation}
  holds for some $c_\textsm{qu} > 0$.
  Let $Y_p^{q,+} \coloneqq \Set{x+q-1}{x \in Y_p}$, $Y_p^{q,-} \coloneqq \Set{x-q}{x \in Y_p}$ and
  \begin{equation*}
    X_k \coloneqq \bigcup_{m=1}^k \big( Y_{k-m+1}^{m,+} \cup Y_{k-m+1}^{m,-} \big),
  \end{equation*}
  so that $n \coloneqq \#X_k = 2 \sum_{m=1}^k \bar{n}_m$.
  If $k \geq c_\textsm{qu} \bar{h}_0^{-1}$ for $\bar{h}_0 \coloneqq \min\{h_0, c_\textsm{qu}^{-1}\}$, then
  \begin{equation} \label{eq:univariate-bound-generic}
    \begin{split} 
      e_{\alpha,\ell}(Q_{\alpha,\ell,X_k}^\textsm{opt}) \leq{}& \exp\bigg(\!-\frac{g(k)^2}{2 \alpha^2} \bigg) \\
      &+ C_1 \sum_{m=k-g(k)+1}^k \exp\Bigg(\!-\bigg[ \frac{(k-m)^2}{2\alpha^2}  + C_2 \bar{n}_m \log(\bar{n}_m) \bigg] \Bigg),
      \end{split}
  \end{equation}
  where $g(k) \coloneqq \floor{k+1-c_\textsm{qu} \bar{h}_0^{-1}}$, the positive constants $C_1$ and $C_2$ are defined in~\eqref{eq:C1C2-constant} and $C$ and $h_0 \leq 1$ are the positive constants in Proposition~\ref{prop:rieger} for $d=1$ and $R=1$.
\end{proposition}

\begin{proof} Let $C$ and $h_0$ be the positive constants of Proposition~\ref{prop:rieger} for $d=1$ and $R=1$ and note that, trivially, $h_0 \leq 1$ since every set has fill-distance of at most one on the unit interval. Consequently, $\bar{h_0} \leq h_0 \leq 1$.
  Define the open intervals
  \begin{equation*}
    \Omega_q^+ \coloneqq (q-1,q) \quad \text{ and } \quad \Omega_q^- \coloneqq (-q, -q+1),
  \end{equation*}
  so that $Y_p^{q,+} \subset \Omega_q^+$ and $Y_p^{q,-} \subset \Omega_q^-$ for all $p,q \in \N$, and $h_{k,m} \coloneqq h_{X_k, \Omega_m^+} = h_{X_k, \Omega_m^-}$. By~\eqref{eq:quasi-uniformity} and the definition of $X_k$ we have
  \begin{equation} \label{eq:quasi-uniformity2}
    h_{k,m} \leq c_\textsm{qu} \bar{n}_{k-m+1}^{-1}
  \end{equation}
  for all $m,k \in \N$ such that $m \leq k$.
  That is, $h_{k,m} \leq h_0 $ when $\bar{n}_{k-m+1} \geq c_\textsm{qu} \bar{h}_0^{-1}$.
  Because $(\bar{n}_m)_{m=1}^\infty$ is a strictly increasing integer sequence such that $\bar{n}_1 \geq 1$, it holds that $\bar{n}_m \geq m$.
  Hence $\bar{n}_{k-m+1} \geq c_\textsm{qu} \bar{h}^{-1}_0$ holds at least when $m \leq g(k) = \floor{k+1-c_\textsm{qu} \bar{h}_0^{-1}}$. Under the assumption $k \geq c_\textsm{qu}\bar{h}_0^{-1}$ we have $g(k) \geq 1$, which means that the sums below are not empty.
  Recall then~\eqref{eq:int-power-function} and decompose the integration domain in the following way:
  \begin{equation} \label{eq:domain-decomposition}
    \begin{split}
      e_{\alpha,\ell}(Q_{\alpha,\ell,X_k}^\textsm{opt}) \leq{}& \frac{1}{\sqrt{2\pi} \alpha} \int_\R P_{\ell,X_k}(x) \exp\bigg(\!-\frac{x^2}{2\alpha^2} \bigg) \dif x \\
    ={}& \sum_{m=1}^{g(k)} \underbrace{ \frac{1}{\sqrt{2\pi} \alpha} \int_{\Omega_m^+ \cup \Omega_m^-} P_{\ell,X_k}(x) \exp\bigg(\!-\frac{x^2}{2\alpha^2} \bigg) \dif x }_{\eqqcolon \varepsilon(m)} \\
    &+ \underbrace{ \frac{1}{\sqrt{2\pi} \alpha} \int_{\R \setminus [-g(k), g(k)]} P_{\ell,X_k}(x) \exp\bigg(\!-\frac{x^2}{2\alpha^2} \bigg) \dif x }_{\eqqcolon \rho(g(k))}.
    \end{split}
  \end{equation}
  To estimate $\varepsilon(m)$, first use the facts that $\exp(-x^2/(2\alpha^2)) \leq \exp(-(m-1)^2/(2\alpha^2))$ on $\Omega_m^+$ and $\Omega_m^-$ and $h_{k,m} = h_{X_k,\Omega_m^+} = h_{X_k,\Omega_m^-}$ and then apply Proposition~\ref{prop:rieger}:
  \begin{equation*}
    \begin{split}
      \varepsilon(m) &\leq \frac{1}{\sqrt{2\pi} \alpha} \exp\bigg(\!-\frac{(m-1)^2}{2\alpha^2} \bigg)  \big( \norm[0]{P_{\ell, X_k}}_{L^1(\Omega_m^+)} + \norm[0]{P_{\ell, X_k}}_{L^1(\Omega_m^-)} \big) \\
        &\leq \frac{\sqrt{2}}{\sqrt{\pi} \alpha} \exp\bigg(\!-\frac{(m-1)^2}{2\alpha^2} \bigg) \exp\big( C \log(h_{k,m}) \, h_{k,m}^{-1} \big).
      \end{split}
  \end{equation*}
  As the function $x \mapsto \log(x)x^{-1}$ is increasing on $(0, \e]$, it follows from~\eqref{eq:quasi-uniformity2} that $\log(h_{k,m}) \, h_{k,m}^{-1} \leq c_\textsm{qu}^{-1} \bar{n}_{k-m+1} [ \log(c_\textsm{qu}) - \log(\bar{n}_{k-m+1})]$ if $\bar{n}_{k-m+1} \geq c_\textsm{qu} \e^{-1}$, which holds at least if $m \leq \floor{k+1 - c_\textsm{qu} \e^{-1}}$ because $\bar{n}_m \geq m$. Since $\bar{h}_0 \leq 1$, this is implied by $m \leq g(k)$. Furthermore, $\log(\bar{n}_{k-m+1}) - \log(c_\textsm{qu}) \geq \frac{1}{2} \log( \bar{n}_{k-m+1} )$ because $\bar{n}_m \geq m$, $m \leq \floor{k+1-c_\textsm{qu} \bar{h}_0^{-1}}$ and $\bar{h}_0 = \min\{h_0, c_\textsm{qu}^{-1}\} \leq c_\textsm{qu}^{-1}$. Hence
  \begin{equation*}
    \begin{split}
      \varepsilon(m) &\leq \frac{\sqrt{2}}{\sqrt{\pi} \alpha} \exp\Bigg(\!-\bigg[ \frac{(m-1)^2}{2\alpha^2}  + C  c_\textsm{qu}^{-1} \bar{n}_{k-m+1} \big[ \log(\bar{n}_{k-m+1}) - \log(c_\textsm{qu})\big] \bigg] \Bigg) \\
      &\leq C_1 \exp\Bigg(\!-\bigg[ \frac{(m-1)^2}{2\alpha^2}  + C_2  \bar{n}_{k-m+1} \log(\bar{n}_{k-m+1}) \bigg] \Bigg)
      \end{split}
  \end{equation*}
when $m \leq g(k)$, where
  \begin{equation} \label{eq:C1C2-constant}
    C_1 \coloneqq \frac{\sqrt{2} }{\sqrt{\pi} \alpha} \quad \text{ and } \quad C_2 \coloneqq \frac{C}{2 c_\textsm{qu}}.
  \end{equation}
  Therefore,
  \begin{equation} \label{eq:err-sum-bound}
    \begin{split}
      \sum_{m=1}^{g(k)} \varepsilon(m) &\leq C_1 \sum_{m=1}^{g(k)} \exp\Bigg(\!-\bigg[ \frac{(m-1)^2}{2\alpha^2}  + C_2 \bar{n}_{k-m+1} \log(\bar{n}_{k-m+1}) \bigg] \Bigg) \\
      &= C_1 \sum_{m=k-g(k)+1}^k \exp\Bigg(\!-\bigg[ \frac{(k-m)^2}{2\alpha^2}  + C_2 \bar{n}_m \log(\bar{n}_m) \bigg] \Bigg).
      \end{split}
  \end{equation}
  Since $P_{\ell,X_k} \leq 1$ by~\eqref{eq:f-pointwise-RKHS}, the remainder term in~\eqref{eq:domain-decomposition} admits the bound
  \begin{equation*}
    \begin{split}
      \rho(g(k)) &= \frac{1}{\sqrt{2\pi} \alpha} \int_{\R \setminus [-g(k), g(k)]} P_{\ell,X_k}(x) \exp\bigg(\!-\frac{x^2}{2\alpha^2} \bigg) \dif x \\
      &\leq \frac{1}{\sqrt{2\pi} \alpha} \int_{\R \setminus [-g(k), g(k)]} \exp\bigg(\!-\frac{x^2}{2\alpha^2} \bigg) \dif x \\
      &= \mathrm{erfc}\bigg( \frac{g(k)}{\sqrt{2} \alpha} \bigg),
      \end{split}
  \end{equation*}
  where $\mathrm{erfc}(x) \coloneqq 2\pi^{-1/2} \int_x^\infty \exp(-t^2) \dif t$ is the complementary error function. Using the standard estimate $\mathrm{erfc}(x) \leq \exp(-x^2)$ we thus obtain the bound
  \begin{equation} \label{eq:rem-bound}
    \rho(g(k)) \leq \exp\bigg(\!-\frac{g(k)^2}{2 \alpha^2} \bigg).
  \end{equation}
  The claim of the theorem follows by inserting the estimates~\eqref{eq:err-sum-bound} and~\eqref{eq:rem-bound} into~\eqref{eq:domain-decomposition}.
\end{proof}

The main result of this section is obtained by selecting the cardinalities of the sets $Y_m$ in Proposition~\ref{prop:kernel-univariate-generic} so as to make derivation of an explicit upper bound feasible.

\begin{theorem} \label{thm:kq-univariate} Consider the point sets $X_k$ in Proposition~\ref{prop:kernel-univariate-generic} and set $\bar{n}_m = m$. Let $\bar{h}_0$ and $c_\textsm{qu}$ be the positive constants in Proposition~\ref{prop:kernel-univariate-generic} and $n = \#X_k$. Then there is a positive constant $C$, which depends only on $\ell$, $\alpha$ and $c_\textsm{qu}$, such that
  \begin{equation} \label{eq:kq-wce-bound-1d}
    e_{\alpha,\ell}( Q_{\alpha,\ell,X_k}^\textsm{opt} ) \leq C  \exp\bigg(\!- \frac{\sqrt{n}}{2 \sqrt{2} \alpha^2} \bigg) 
  \end{equation}
  whenever $k \geq 2 c_\textsm{qu} \bar{h}_0^{-1}$. 
\end{theorem}
\begin{proof} With $\bar{n}_m = m$ we have $n = \#X_k = 2 \sum_{m=1}^k \bar{n}_m = 2 \sum_{m=1}^k m = k(k+1) \leq 2k^2$.
  Let $\bar{h}_0$ and $c_\textsm{qu}$ be the positive constants from Proposition~\ref{prop:kernel-univariate-generic} and suppose that $k$ is large enough that $k-c_\textsm{qu} \bar{h}_0^{-1} \geq k / 2$ (i.e., $k \geq 2c_\textsm{qu} \bar{h}_0^{-1}$). Then
  \begin{equation*}
    \begin{split}
      \exp\bigg(\!-\frac{g(k)^2}{2 \alpha^2} \bigg) = \exp\bigg(\!-\frac{(\floor{k+1-c_\textsm{qu} \bar{h}_0^{-1}})^2}{2 \alpha^2} \bigg) &\leq \exp\bigg(\!-\frac{(k-c_\textsm{qu} \bar{h}_0^{-1})^2}{2 \alpha^2} \bigg) \\
      &\leq \exp\bigg(\! - \frac{k^2}{8 \alpha^2} \bigg) \\
      &\leq \exp\bigg(\! - \frac{n}{16 \alpha^2} \bigg).
      \end{split}
  \end{equation*}
  Furthermore,
  \begin{equation*}
    \begin{split}
      \sum_{m=k-g(k)+1}^k &\exp\Bigg(\!-\bigg[ \frac{(k-m)^2}{2\alpha^2}  + C_2 \bar{n}_m \log(\bar{n}_m) \bigg] \Bigg) \\
      &\leq \sum_{m=1}^k \exp\Bigg(\!-\bigg[ \frac{(k-m)^2}{2\alpha^2}  + C_2 \bar{n}_m \log(\bar{n}_m) \bigg] \Bigg) \\
      &\leq \sum_{m=1}^k \exp\Bigg(\!-\bigg[ \frac{k-m^2}{2\alpha^2}  + C_2 \bar{n}_m \log(\bar{n}_m) \bigg] \Bigg) \\
      &\leq \exp\bigg(\!-\frac{k}{2\alpha^2} \bigg) \sum_{m=1}^k \exp\Bigg(\!-\bigg[ C_2 \bar{n}_m \log(\bar{n}_m) -\frac{m}{2\alpha^2} \bigg] \Bigg) \\
      &\leq \exp\bigg(\!-\frac{k}{2\alpha^2} \bigg) \sum_{m=1}^k \exp\Bigg(\!-m\bigg[ C_2 \log(m) -\frac{1}{2\alpha^2} \bigg] \Bigg).
      \end{split}
  \end{equation*}
  Because the exponent in the sum is negative when $\log(m) \geq (2C_2\alpha^2)^{-1}$ and for such $m$ the terms in the sum decay super-exponentially, we conclude that there is $C_3 > 0$, which depends only on $\ell$, $\alpha$ and $c_\textsm{qu}$, such that
  \begin{equation} \label{eq:univariate-bound2}
    \begin{split}
      \sum_{m=k-g(k)+1}^k \exp\Bigg(\!-\bigg[ \frac{(k-m)^2}{2\alpha^2}  + C_2 \bar{n}_m \log(\bar{n}_m) \bigg] \Bigg) &\leq C_3 \exp\bigg(\!-\frac{k}{2\alpha^2}\bigg) \\
      &\leq C_3 \exp\bigg(\!-\frac{\sqrt{n}}{2\sqrt{2} \alpha^2} \bigg) .
      \end{split}
  \end{equation}
  Upon insertion of the estimates above into~\eqref{eq:univariate-bound-generic} it is seen that~\eqref{eq:univariate-bound2} dominates the estimate. This yields the claim.
\end{proof}

The bound~\eqref{eq:kq-wce-bound-1d} is worse than the bound~\eqref{eq:gh-1d} for scaled Gauss--Hermite rules and the bounds obtained in \citet{KuoWozniakowski2012} and \citet{KuoSloanWozniakowski2017} for standard Gauss--Hermite rules.
We partly attribute this to the sub-optimal selection, done out of convenience, of the points $X_k$;
given that the Gaussian weight function decays super-exponentially, one would expect that the points should be more concentrated at the origin.
Moreover, the bound~\eqref{eq:rieger} on which the results are based is potentially sub-optimal and the locally quasi-uniform point sets we are using are likely not suitable for approximating analytic functions~\citep{Platte2005,Platte2011,PlatteTrefethen2011}.
As is evident from Figure~\ref{fig:wce-2d}, the estimates used in the proofs of Proposition~\ref{prop:kernel-univariate-generic} and Theorem~\ref{thm:kq-univariate} appear to be somewhat rough.
Nevertheless, this second integration rule we have proposed enjoys substantial flexibility with respect to the choice of the point set, in particular it admits sequences of nested point sets for an extensible treatment.

\begin{figure}[th!]
  \centering
  \includegraphics{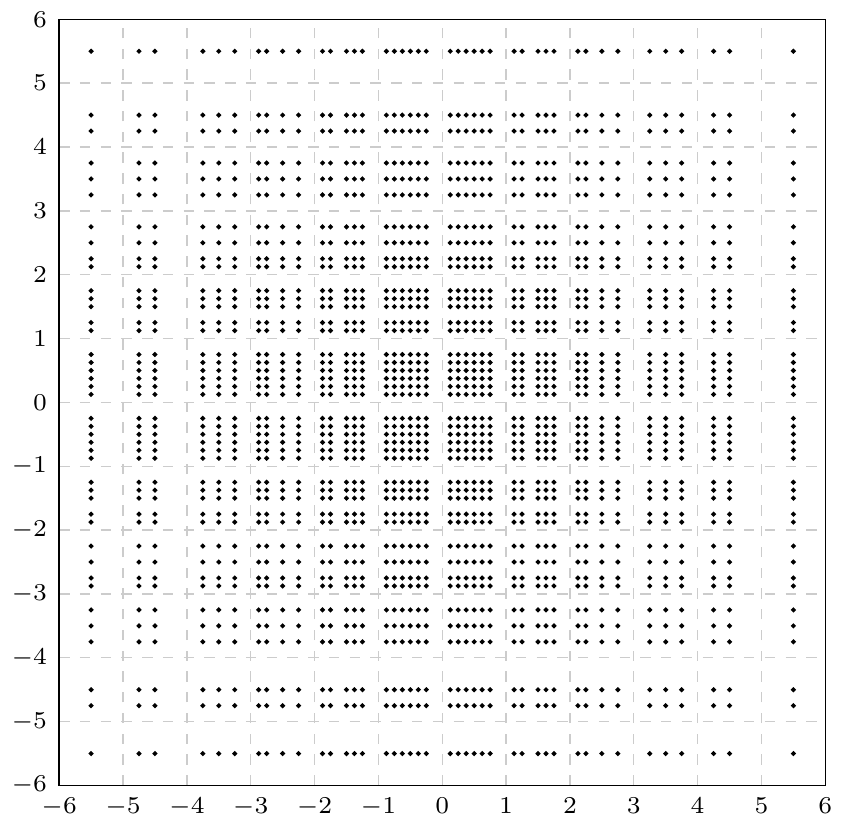}
  \caption{The point set $X_{\b{k}}$ in $\R^2$ with $\b{k} = (6,6)$ and $Y_m$ consisting of the first $m$ points in the van der Corput sequence. The total number of points is 1764.}
  \label{fig:vdc-2d}
\end{figure}

\subsection{Error estimates for tensor product rules} \label{sec:kernel-quad-multi}

In this section we consider the multivariate Gaussian kernel~\eqref{eq:gauss-kernel}, with length-scale parameter $\b{\ell}$.
Let $X_k$ be the point sets constructed in Theorem~\ref{thm:kq-univariate}.
For $\b{k} \in \N^d$ define the product grid
\begin{equation} \label{eq:tensor-grid-kq}
  X_{\b{k}} \coloneqq X_{k_1} \times \cdots \times X_{k_d}.
\end{equation}
This set consists of $N \coloneqq \#X_{\b{k}} = \prod_{i=1}^d \#X_{k_i} = \prod_{i=1}^d k_i(k_i+1)$ points.

\begin{theorem} \label{thm:kq-tensor} Consider the product grid $X_{\b{k}}$ defined in~\eqref{eq:tensor-grid-kq}. Let $\bar{h}_0$ and $c_\textsm{qu}$ be the positive constants in Proposition~\ref{prop:kernel-univariate-generic}. and $n_i \coloneqq \#X_{k_i}$. Then, for $i=1,\ldots,d$, there are positive constants $C_i$, each of which only depends on $\ell_i$, $\alpha_i$ and $c_\textsm{qu}$, such that
  \begin{equation*}
    e_{\b{\alpha}, \b{\ell}}( Q_{\b{\alpha}, \b{\ell}, X_{\b{k}}}^\textsm{opt} ) \leq \sum_{i=1}^d C_i \Bigg[ \prod_{j \neq i} \bigg(1 + \frac{2\alpha_j^2}{\ell_j^2} \bigg)^{-1/4} \Bigg]  \exp\bigg(\! -\frac{\sqrt{n_i}}{2\sqrt{2}\alpha_i^2} \bigg) 
  \end{equation*}
  whenever $k_i \geq 2 c_\textsm{qu} \bar{h}_0^{-1}$ for every $i=1,\ldots,d$.
\end{theorem}
\begin{proof} Proceeding as in the proof of Theorem~\ref{thm:gh-tensor} yields
    \begin{equation*}
    \begin{split}
      e_{\b{\alpha}, \b{\ell}}( Q_{\b{\alpha}, \b{\ell}, X_{\b{k}}}^\textsm{opt} ) &\leq \sum_{i=1}^d e_{\alpha_i,\ell_i}( Q_{\alpha_i,\ell_i,X_{k_i}}^\textsm{opt} ) \norm[0]{\mathcal{I}_{\b{\alpha}, \b{\ell}}(i)}_{\b{\ell}(1:i)} \prod_{j=i+1}^{d} \norm[0]{ \mathcal{Q}_{\alpha_j,\ell_j,X_{k_j}}^\textsm{opt} }_{\ell_j}.
      \end{split}
    \end{equation*}
    From~\eqref{eq:wce-kq} it follows that $\norm[0]{ \mathcal{Q}_{\alpha,\ell,X}^\textsm{opt} }_{\ell} \leq \norm[0]{\mathcal{I}_{\alpha,\ell}}_\ell$ for any $\alpha, \ell > 0$ and any point set $X \subset \R$.
    Estimates in Theorem~\ref{thm:kq-univariate} and Lemma~\ref{lemma:representers} for the worst-case errors and the norms of the integral representers, respectively, yield the claim.
\end{proof}

\begin{figure}[t!]
  \centering
  \includegraphics{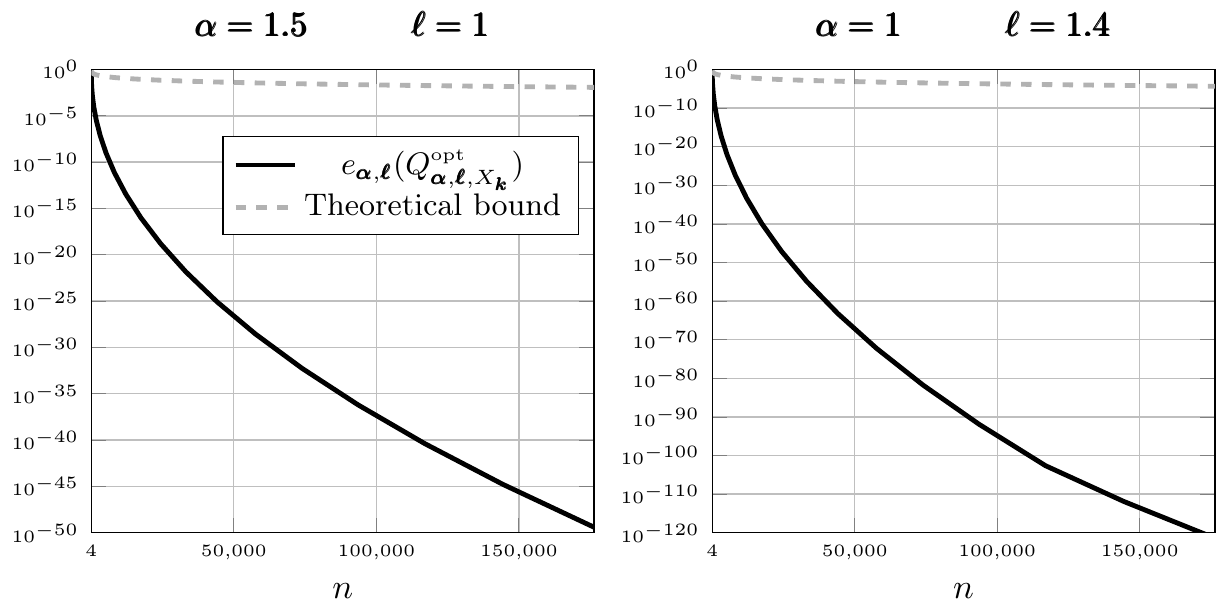}
  \caption{True worst-case error and the upper bound~\eqref{eq:kq-isotropic-bound} (with $C=1$) in the isotropic setting for $d=2$ and $k=1,\ldots,20$. As in Figure~\ref{fig:vdc-2d}, $Y_m$ consists of the $m$ first in the van der Corput sequence. All computations were implemented in Python with 400-digit precision.} \label{fig:wce-2d}
\end{figure}

In the isotropic case the statement simplifies to the statement in Corollary \ref{cor:kq-tensor}:

\begin{corollary} \label{cor:kq-tensor} Consider the product grid $X_{\b{k}}$ defined in~\eqref{eq:tensor-grid-kq}. Let \sloppy{${\alpha_1 = \cdots = \alpha_d = \alpha}$}, $\ell_1 = \cdots = \ell_d = \ell$ and \sloppy{${k_1 = \cdots = k_d = k}$} for $\alpha,\ell>0$ and $k \geq 1$. Let $\bar{h}_0$ and $c_\textsm{qu}$ be the positive constants in Proposition~\ref{prop:kernel-univariate-generic} and $n = \#X_k = k(k+1)$. Then there is a positive constant $C$, which only depends on $\ell$, $\alpha$, $d$ and $c_\textsm{qu}$, such that
  \begin{equation} \label{eq:kq-isotropic-bound}
    e_{\b{\alpha}, \b{\ell}}( Q_{\b{\alpha}, \b{\ell}, X_{\b{k}}}^\textsm{opt} ) \leq C  \exp\bigg(\! -\frac{\sqrt{n}}{2\sqrt{2}\alpha^2} \bigg) 
  \end{equation}
  whenever $k \geq 2c_\textsm{qu} \bar{h}_0^{-1}$.
\end{corollary}

Because $N = \#X_{\b{k}} = n^d$, in terms of the total number of points this bound is
\begin{equation*}
  e_{\b{\alpha}, \b{\ell}}( Q_{\b{\alpha}, \b{\ell}, X_{\b{k}}}^\textsm{opt} ) \leq C \,  \exp\bigg(\! -\frac{N^{1/(2d)}}{2\sqrt{2}\alpha^2} \bigg) ,
\end{equation*}
which, like ~\eqref{eq:Nd-GH}, shows that for large $d$ one should expect slower convergence.
Figure~\ref{fig:wce-2d} shows that the above error bounds are very conservative.

\begin{remark}
    To the best of our knowledge no lower bounds from which a counterpart to Proposition~\ref{prop:rieger} could be derived have been established. The only lower bound we can supply follows from Theorem~\ref{thm:nth-minimal-d}. In the isotropic setting of Corollary~\ref{cor:kq-tensor} with $N+1 = (n+1)^d$ the bound is
    \begin{equation*}
      e_{\b{\alpha}, \b{\ell}}( Q_{\b{\alpha}, \b{\ell}, X_{\b{k}}}^\textsm{opt} ) \geq \bar{C}_n(\gamma)^d (N+1)^{-1} \bigg( \frac{\omega_{\gamma} \e}{4n} \bigg)^{dn} \geq \bar{C}_n(\gamma)^d c^{-(N+1)^{1/d}} (N+1)^{-((N+1)^{1/d}+1)},
    \end{equation*}
    where $c = (4/(\omega_\gamma \e))^d > 1$ and $\gamma$, $\omega_\gamma$ and $\bar{C}_n(\gamma)$ are defined in Theorem~\ref{thm:nth-minimal-1d}.
\end{remark}

\section{Conclusions and discussion}
 
We constructed two classes of integration rules for integration of functions in reproducing kernel Hilbert spaces of Gaussian kernels defined on $\R^d$.
For the first class of methods, those based on suitable scaling of Gauss--Hermite rules, we derived upper and lower bounds on the worst-case integration error. In dimension $d$, the lower bounds are of the form $\exp(-c_1 N^{1/d}) N^{1/(4d)}$ and upper bounds of the form $\exp(-c_2 N^{1/d}) N^{-1/(4d)}$, where $N$ is the total number of points and $c_1 > c_2$ are positive constants.
In contrast to integration rules analysed in previous work, the bounds are valid for any variance parameter of the integration density and length-scale parameter of the kernel.
Our second construction used optimal weights for points that can be taken as a nested sequence. In this case we proved an upper bound for the worst-case error of the form $\exp(-c_3 N^{1/(2d)})$ for a constant $c_3 > 0$.
Several improvements and extensions are possible:
\begin{itemize}
\item As observed in Remark~\ref{remark:gh-suboptimality} and Figure~\ref{fig:wce-gh-scaled}, there is room for improvement in the upper and lower bounds for the worst-case error of a scaled Gauss--Hermite rule.
\item Extending the construction and error estimates in Section~\ref{sec:scaled-GH} for general weight functions would be interesting, but explicit error estimates may be more difficult to derive; see Remark~\ref{remark:general-weights}.
\item The point sets used in Theorem~\ref{thm:kq-univariate} and its tensor product extensions are likely sub-optimal, placing too many points away from the origin, where most of the probability mass is located, and being locally too uniform. We believe that it may be possible to derive exponential rates of convergence for this construction if the points are placed more carefully.
\item It is clear that the domain decomposition technique used to prove Proposition~\ref{prop:kernel-univariate-generic} and Theorem~\ref{thm:kq-univariate} can be used also in higher dimensions, circumventing the need for restrictive product grids. However, decomposition into sub-domains that are not translations of one another may be necessary, and this requires more careful handling of the constants $C$ and $h_0$ in Proposition~\ref{prop:rieger} or its generalisation for general domains~\citet{RiegerZwicknagl2010} and the constant $c_\textsm{qu}$ in~\eqref{eq:quasi-uniformity}.
\item The point selection and error analysis in Section~\ref{sec:kernel-quadrature} are not intrinsically related to the Gaussian kernel and weight function. Other kernels for which results similar to Proposition~\ref{prop:rieger} have been proved, such as those inducing Sobolev spaces, could be used instead.
\item As has been noted, various tractability results could be proved following~\citet{KuoSloanWozniakowski2017}.
\end{itemize}

\section*{Acknowledgements}

The authors were supported by the Lloyd's Register Foundation programme on data-centric engineering at the Alan Turing Institute, United Kingdom.
The authors are grateful to the reviewers for their suggestions and comments that led to sharper upper bounds.

\end{document}